\documentclass{amsart}
\usepackage[T1]{fontenc}
\usepackage{amsfonts}
\usepackage{mathrsfs}
\usepackage{mathtools}
\usepackage{amssymb}
\usepackage[colorlinks=true,linkcolor=blue,allcolors=blue
]{hyperref}
\usepackage[noabbrev,capitalize,nameinlink]{cleveref}
\usepackage{amsmath,blkarray,booktabs,bigstrut}
\usepackage{cite}
\usepackage{graphicx}
\usepackage{dutchcal}
\usepackage{tikz}
\usepackage{tikz-cd}
\numberwithin{equation}{section}

\newcommand\cP{{\mathcal P}}

\DeclareMathOperator{\Ker}{Ker}
\DeclareMathOperator{\Ran}{Ran}
\DeclareMathOperator{\supp}{supp}
\DeclareMathOperator{\spanning}{span}

\DeclareMathOperator{\trace}{trace}

\theoremstyle{plain}
  \newtheorem{theorem}{Theorem}[section]
  \newtheorem{proposition}[theorem]{Proposition}
  \newtheorem{lemma}[theorem]{Lemma}
  \newtheorem{corollary}[theorem]{Corollary}

  \newtheorem{example}{Example}[section]

\theoremstyle{definition}
  \newtheorem{definition}[theorem]{Definition}

\begin{document}
\include{psfig}
\title[Probabilistic Dual Frames and Optimal Transport]{Probabilistic Dual Frames and Minimization of Dual Frame Potentials}

\author{Dongwei Chen}
\thanks{Dongwei Chen, dongwei.chen@colostate.edu, ORCID: 0000-0003-1956-7683}
\thanks{Department of Mathematics, Colorado State University, Fort Collins, CO, USA, 80523}

\begin{abstract}
This paper studies probabilistic dual frames and the associated dual frame potentials from the perspective of optimal mass transport. The main contribution of this work shows that given a probabilistic frame, its associated dual frame potential is minimized if and only if the probabilistic frame is tight and the probabilistic dual frame is the canonical dual. In particular, the tightness condition can be dropped if the probabilistic dual frame potential is minimized only among probabilistic dual frames of pushforward type.
\end{abstract}

\maketitle

\section{Introduction and Main Results}\label{introduction}
A sequence of vectors $\{{\bf f}_i\}_{i=1}^\infty$ in a separable Hilbert space $\mathcal{H}$ is called a \textit{frame} if  there exist $0<A \leq B < \infty$ such that for any ${\bf f} \in \mathcal{H}$, 
\begin{equation*}
    A \Vert {\bf f}  \Vert^2 \leq\sum_{i =1}^{\infty} \vert \left\langle {\bf f},{\bf f}_i \right\rangle \vert ^2  \leq B\Vert {\bf f} \Vert^2.
\end{equation*} 
Furthermore, $\{{\bf f}_i\}_{i=1}^\infty$ is called a tight frame if $A = B$ and Parseval if $A = B=1$. Frames were first introduced by Duffin and Schaeffer in the context of nonharmonic Fourier analysis \cite{duffin1952class} and have been applied in many areas, such as the Kadison-Singer problem \cite{casazza2013kadison},  time-frequency analysis \cite{grochenig2001foundations}, and wavelet analysis \cite{daubechies1992ten}.  As the extension of an orthonormal basis, a frame permits linear dependence between frame elements and allows each vector in $\mathcal{H}$ to be written as a linear combination of frame elements in a redundant way.  
A sequence of vectors $\{{\bf g}_i\}_{i=1}^\infty$ in $\mathcal{H}$ is a \textit{dual frame} for the frame $\{{\bf f}_i\}_{i=1}^\infty$ if for any ${\bf f} \in \mathcal{H}$, 
\begin{equation*}
   {\bf f} = \sum_{i =1}^{\infty}  \left\langle {\bf f},{\bf g}_i \right\rangle  {\bf f}_i   = \sum_{i =1}^{\infty}  \left\langle {\bf f},{\bf f}_i \right\rangle  {\bf g}_i.
\end{equation*}
An example of dual frames is the canonical dual $\{{\bf S}^{-1}{\bf f}_i\}_{i=1}^\infty$ where ${\bf S}: \mathcal{H} \rightarrow \mathcal{H}$ is the \textit{frame operator} given by $ {\bf S}({\bf f}) = \sum_{i =1}^{\infty}  \left\langle {\bf f},{\bf f}_i \right\rangle  {\bf f}_i $, $ {\bf f} \in \mathcal{H}$. 
Indeed, ${\bf S}: \mathcal{H} \rightarrow \mathcal{H}$ is bounded, positive, self-adjoint, and invertible with a bounded inverse.
See \cite{christensen2016introduction} for more details on frames.

Let $\{{\bf y}_i\}_{i=1}^N $ be a (finite) frame in the Euclidean space $\mathbb{R}^n$ with frame bounds $0 <A \leq B$ and $\mu_f := \frac{1}{N}  \sum\limits_{i=1}^N  \delta_{{\bf y}_i}$ the uniform distribution on $\{{\bf y}_i\}_{i=1}^N $ where $N \geq n$. Then for any ${\bf x} \in \mathbb{R}^n$,
\begin{equation*}
    \frac{A}{N} \Vert {\bf x}  \Vert^2 \leq\int_{\mathbb{R}^n} \vert \left\langle {\bf x},{\bf y} \right\rangle \vert ^2 d\mu_f({\bf y})  = \frac{1}{N}\sum_{i =1}^{N} \vert \left\langle {\bf x},{\bf y}_i \right\rangle \vert ^2  \leq \frac{B}{N} \Vert {\bf x}  \Vert^2.
\end{equation*}
In this case, finite frames can be viewed as discrete probabilistic measures on $\mathbb{R}^n$. Inspired by this insight, Martin Ehler and Kasso A. Okoudjou in \cite{ehler2012random, ehler2013probabilistic} introduced the concept of probabilistic frame, which is a probability measure on $\mathbb{R}^n$ satisfying a so-called frame inequality. Throughout the paper, let $\mathcal{P}(\mathbb{R}^n)$ be the set of Borel probability measures on $\mathbb{R}^n$ and  for $\mu \in \mathcal{P}(\mathbb{R}^n)$, we define its second moment as 
$$M_2(\mu) := \int_{\mathbb{R}^n} \Vert {\bf x} \Vert ^2 d\mu({\bf x}).$$
Furthermore, let
$\mathcal{P}_2(\mathbb{R}^n) \subset \mathcal{P}(\mathbb{R}^n)$ be the set of probability measures with finite second moments, that is, 
\begin{equation*}
\mathcal{P}_2(\mathbb{R}^n) := \left \{ \mu \in \mathcal{P}(\mathbb{R}^n):  \int_{\mathbb{R}^n} \Vert {\bf x} \Vert ^2 d\mu({\bf x}) < + \infty \right \}.
\end{equation*}
In addition, the \textit{support} of $\mu \in  \mathcal{P}(\mathbb{R}^n)$ is given by 
$$\supp(\mu) := \{ {\bf x} \in \mathbb{R}^n: \text{for any} \ r>0, \mu(B_r({\bf x}))>0  \},$$
where $B_r({\bf x})$ is an open ball centered at ${\bf x}$ with radius $r>0$.
If $\mu \in  \mathcal{P}(\mathbb{R}^n) $ and $f: \mathbb{R}^n \rightarrow \mathbb{R}^m$ is a Borel measurable map where $n$ may differ from $m$, then $f_{\#} \mu \in \mathcal{P}(\mathbb{R}^m)$ is called the \textit{pushforward} of $\mu$ by  the map $f$,  and is defined as
$$f_{\#} \mu (E) := (\mu \circ f^{-1}) (E) = \mu \big (f^{-1}(E) \big ) \ \text{for any Borel set}  \ E \subset \mathbb{R}^m.$$ If $f$ is linear and represented by a matrix ${\bf A}$ with respect to some basis, then ${\bf A}_{\#} \mu$ is used to denote $f_{\#} \mu$. In particular, $(\mathbf{Id}, f)$ is used to denote the map from $\mathbb{R}^n$ to $\mathbb{R}^n \times \mathbb{R}^m$ via $\mathbf{x} \mapsto (\mathbf{x}, f(\mathbf{x}))$. In such a case, $(\mathbf{Id}, f)_\#\mu$ is a probability measure on $\mathbb{R}^n \times \mathbb{R}^m$ that is supported on the graph of $f$.  Then, we have the following definition for probabilistic frames. 

\begin{definition}
$\mu \in  \mathcal{P}(\mathbb{R}^n)$ is said to be a \textit{probabilistic frame} for $\mathbb{R}^n$ if there exist $0<A \leq B < \infty$ such that for any ${\bf x} \in \mathbb{R}^n$, 
\begin{equation*}
    A \Vert {\bf x}  \Vert^2 \leq\int_{\mathbb{R}^n} \vert \left\langle {\bf x},{\bf y} \right\rangle \vert ^2 d\mu({\bf y})  \leq B \Vert {\bf x}  \Vert^2.
\end{equation*}
$\mu$ is called a tight probabilistic frame if $A = B$ and Parseval if $A = B=1$. Moreover, $\mu$ is said to be a Bessel probability measure if only the upper bound holds. 
\end{definition}

Clearly if $\mu \in  \mathcal{P}_2(\mathbb{R}^n)$, $\mu$ is Bessel with bound $M_2(\mu) $. One can also define the \textit{frame operator} ${\bf S}_\mu$ for a Bessel probability measure $\mu$  as the following matrix:
    \begin{equation*}
          {\bf S}_\mu := \int_{\mathbb{R}^n}  {\bf y} {\bf y}^t d \mu({\bf y}).
    \end{equation*}
Probabilistic frames can be characterized by frame operators as below.   
 \begin{proposition}[Theorem 12.1 in \cite{ehler2013probabilistic}, Proposition 3.1 in \cite{maslouhi2019probabilistic}] \label{TAcharacterization}
Let $\mu  \in \mathcal{P}(\mathbb{R}^n)$. 
\begin{itemize}
    \item[$(1)$] $\mu$ is a probabilistic frame $\Leftrightarrow$ ${\bf S}_{\mu}$ is positive definite $\Leftrightarrow$ $\mu \in \mathcal{P}_2(\mathbb{R}^n)$ and $\spanning\{\supp(\mu)\} = \mathbb{R}^n$.
   
     \item[$(2)$]  $\mu$ is a tight probabilistic frame with bound $A > 0$ $\Leftrightarrow$ ${\bf S}_{\mu} = A \ {\bf Id}$ where ${\bf Id}$ is the identity matrix of size $n \times n$. 
\end{itemize}
\end{proposition}

Since probabilistic frames have finite second moments, they naturally lie in $\mathcal{P}_2(\mathbb{R}^n)$ and can therefore be studied using optimal transport and the 2-Wasserstein distance \cite{ehler2013probabilistic, wickman2014optimal}. Given two probabilistic frames $\mu$ and $\nu$, let $\Gamma(\mu,\nu)$ be the set of transport couplings with marginals $\mu$ and $\nu$, that is,
$$ \Gamma(\mu,\nu) :=  \left \{ \gamma \in \mathcal{P}(\mathbb{R}^n \times \mathbb{R}^n): {\pi_{{ x}}}_{\#} \gamma = \mu, \ {\pi_{{ y}}}_{\#} \gamma = \nu \right \},$$
where $\pi_{{ x}}$, $\pi_{{ y}}$ are projections onto the ${\bf x}$ and ${\bf y}$ coordinates: for any $({\bf x},{\bf y}) \in \mathbb{R}^n \times \mathbb{R}^n $, $\pi_{{ x}}({\bf x}, {\bf y}) = {\bf x}$ and $ \pi_{{ y}}({\bf x},{\bf y}) = {\bf y}$. The 2-Wasserstein distance $ W_2(\mu,\nu)$ is often used to quantify the distance between probabilistic frames $\mu$ and $\nu$:  
$$ W_2(\mu,\nu) := \left (\underset{\gamma \in \Gamma(\mu,\nu)}{\text{inf}}  \int_{\mathbb{R}^n \times \mathbb{R}^n}  \left \Vert {\bf x}-{\bf y}\right \Vert^2 \ d\gamma({\bf x},{\bf y}) \right)^{1/2}.$$
This is of interest even if one is only concerned with finite frames, since the $2$-Wasserstein distance can quantify the distance between probabilistic frames induced by finite frames of different cardinalities.
For interested readers, we refer to \cite{ehler2013probabilistic, wickman2014optimal, ehler2012random, ehler2012minimization, chen2023paley,chen2025probabilistic} for more details on probabilistic frames, and \cite{figalli2021invitation} for optimal transport.

In physics, engineering, and signal processing, it is a common situation to have a measure
space $(\Omega,\mu)$, a Hilbert space $\mathcal{H}$, and a measurable map $\Psi: \Omega \rightarrow \mathcal{H}$ given by $ t  \mapsto \Psi_t$ such that
\begin{equation*}
\|{\bf f}\|^2 = \int_{\Omega} |\langle {\bf f}, \Psi_t\rangle|^2 d\mu(t), \  \text{for all } {\bf f} \in \mathcal{H}.
\end{equation*}
Such a measurable map is called a coherent state, and the Gabor transform on $L^2(\mathbb{R})$ and Fourier transform on $L^2(\mathbb{R})$ (when restricted to a set of measure $2\pi$) are such examples. Modern frame theory was initiated as a method of discretizing coherent states, and all the examples of frames given in the landmark paper \cite{daubechies1986painless} were created by sampling different coherent states on lattices. Coherent states were generalized to continuous frames by relaxing the equality to an inequality \cite{ali1993continuous}: a measurable map $\Psi: \Omega \rightarrow \mathcal{H}$ given by $ t  \mapsto \Psi_t$ is called a \emph{continuous frame} of $\mathcal{H}$ if there exist constants $0 < A \le B < \infty$ so that
\begin{equation*}
A \|{\bf f}\|^2 \le \int_{\Omega} |\langle {\bf f}, \Psi_t\rangle|^2  d\mu(t) \le B\|{\bf f}\|^2, \  \text{for all } {\bf f} \in \mathcal{H}.
\end{equation*}

We can think of a probabilistic frame as the distribution (or pushforward measure) on $\mathbb{R}^n$ of a continuous frame on a probability measure space. 
\begin{proposition}[Proposition 1.12 and discussion in \cite{wickman2014optimal}] 
    Suppose $\mu \in  \mathcal{P}(\mathbb{R}^n)$ is a probabilistic frame. Then the identity map $\mathbf{Id}: \mathbb{R}^n \rightarrow \mathbb{R}^n$
    is a continuous frame for $\mathbb{R}^n$ regarding $\mu$ with the same bounds as $\mu$. Conversely, if $\Psi: \Omega \rightarrow \mathbb{R}^n$ is a continuous frame for $\mathbb{R}^n$ with bounds $0<A \leq B$ regarding $(\Omega, \mu)$ where $\mu(\Omega) < \infty$, then $\Psi_\#(\frac{\mu}{\mu(\Omega)})$ is a probabilistic frame for $\mathbb{R}^n$ with bounds $\frac{A}{\mu(\Omega)}$ and $\frac{B}{\mu(\Omega)}$.
\end{proposition}
Although continuous frames are well studied, there are numerous advantages to considering probabilistic frames specifically in the context of the frame potential. This is because the frame potential only cares about the probability distribution of a continuous frame (namely, how frame vectors are spread out in $\mathcal{H}$) and not the particular indexing of $t \mapsto \Psi_t$.

One of the main contributions of this work is about the minimization of probabilistic dual frame potentials. The energy minimization problem is an active topic in mathematics and people in the frame theory community are particularly interested in the minimization of frame potentials \cite{benedetto2003finite, ehler2012minimization, christensen2020equiangular, aceska2022cross,wickman2023gradient}. Given a unit-norm frame $\{{\bf f}_i\}_{i=1}^N \subset \mathbb{S}^{n-1}$ where $N \geq n$, its frame potential (FP) is defined as 
$$\text{FP}(\{{\bf f}_i\}_{i=1}^N ) := \sum_{i=1}^N \sum_{j=1}^N |\langle {\bf f}_i, {\bf f}_j \rangle |^2.$$
John J. Benedetto and Matthew Fickus in \cite{benedetto2003finite} showed that  
$\text{FP}(\{{\bf f}_i\}_{i=1}^N ) \geq \frac{N^2}{n},$
and the equality holds if and only if the unit-norm frame $\{{\bf f}_i\}_{i=1}^N$ is tight. Similarly, the authors in \cite{ehler2012minimization, wickman2023gradient} defined the probabilistic frame potential (PFP) for a given probability measure $\mu \in \mathcal{P}_2(\mathbb{R}^n)$ and found its lower bound: 
$$\text{PFP}(\mu) := \int_{\mathbb{R}^n} \int_{\mathbb{R}^n} |\langle {\bf x}, {\bf y} \rangle|^2 d\mu({\bf x}) d\mu({\bf y}) \geq \frac{M_2^2(\mu)}{n},$$
and if $\mu \neq \delta_{{\bf 0}}$, the equality holds if and only if $\mu$ is a tight probabilistic frame.

The main results given by Theorem \ref{DualFramePotential} and Theorem \ref{dualframepotential2} in this paper say that given a probabilistic frame $\mu$ with bounds $0<A \leq B$ and its probabilistic dual frame $\nu$, the probabilistic dual frame potential (PDFP) between $\mu$ and $\nu$ satisfies
          \begin{equation*}
              \text{PDFP}_\mu(\nu) := \int_{\mathbb{R}^n}  \int_{\mathbb{R}^n} |\langle {\bf x}, {\bf y} \rangle |^2 d\mu({\bf x}) d\nu({\bf y}) \geq  \frac{A}{B} n,
          \end{equation*}
and the equality holds if and only if $\mu$ is a tight probabilistic frame and $\nu $ is the canonical dual $ {{\bf S}^{-1}_\mu}_\# \mu$. Moreover, the tightness condition can be dropped: if $T_{\#}\mu$ is a probabilistic dual frame to $\mu$ where the map $T: \mathbb{R}^n \rightarrow \mathbb{R}^n$ is measurable, then 
    \begin{equation*}
              \text{PDFP}_\mu(T_{\#}\mu):= \int_{\mathbb{R}^n}  \int_{\mathbb{R}^n}  |\langle {\bf x}, T({\bf y}) \rangle|^{2} d\mu({\bf x}) d\mu({\bf y}) \geq n,
    \end{equation*}
and the equality holds if and only if $T({\bf y}) =  {\bf S}_\mu^{-1}{\bf y}$ for $\mu$ almost all ${\bf y} \in \mathbb{R}^n$, which means that the probabilistic dual frame $T_{\#}\mu$ is indeed the canonical dual ${{\bf S}_\mu^{-1}}_\#\mu$ up to a $\mu$-null set. Therefore, the canonical dual is the minimizer of the probabilistic dual frame potential among all probabilistic dual frames of pushforward type.

This paper is organized as follows. In Section \ref{dualframe}, we prove several properties of probabilistic dual frames, and we especially focus on the probabilistic dual frames of pushforward type. In Section \ref{dualpotential}, we give a lower bound for the probabilistic dual frame potential, which is attained if and only if the given probabilistic frame is tight and the associated probabilistic dual frame is the canonical dual. And the tightness assumption can be dropped if the dual frame potential is minimized only among probabilistic dual frames of pushforward type. In Section \ref{section4}, we discuss related open questions about probabilistic dual frame potentials. Finally, in Section \ref{ProofOfLemmas}, we give proofs of Proposition \ref{denseProperty} and Proposition \ref{dualperturbation}.

\section{Probabilistic Dual Frames}\label{dualframe}

This section starts with probabilistic dual frames. Throughout the paper,  we use ${\bf A}$ and ${\bf x}$ for a real matrix and a vector in $\mathbb{R}^n$, and $A$ and $x$ for real numbers. Especially, ${\bf 0}$ is used to denote the zero vector in $\mathbb{R}^n$, ${\bf 0}_{n \times n}$ the $n \times n$ zero matrix, and $0$ the real number zero. We use ${\bf Id}$ for the $n \times n$ identity matrix, $\text{Id}$ for the identity map on $\mathbb{R}$, and  ${\bf A}^t$ for the transpose of the matrix ${\bf A}$. 

\begin{definition}
Let $\mu$ be a probabilistic frame. $\nu \in \mathcal{P}_2(\mathbb{R}^n)$ is called a \textit{probabilistic dual frame} of $\mu$ if there exists $\gamma \in \Gamma(\mu, \nu)$ such that
\begin{equation*}
  \int_{\mathbb{R}^n \times \mathbb{R}^n} {\bf x}{\bf y}^t d\gamma({\bf x, y}) = {\bf Id}.
\end{equation*}
\end{definition}

Probabilistic dual frames are also known as transport duals, and it was shown that every probabilistic dual frame is a probabilistic frame \cite{wickman2017duality}. If $\mu$ is a probabilistic frame with bounds $0<A \leq B$, then ${{\bf S}_\mu^{-1}}_{\#}\mu$ is called the canonical probabilistic dual frame of $\mu$ (with bounds $\frac{1}{B}$ and $ \frac{1}{A}$), since one can let $\gamma:= {({\bf Id}, {\bf S}_\mu^{-1})}_{\#}\mu \in \Gamma(\mu, {{\bf S}_\mu^{-1}}_{\#}\mu)$. We also have the following equivalence for probabilistic dual frames. 
 
\begin{lemma}\label{DualFrameEquivalence}
    Let $\mu$ be a probabilistic frame. Suppose $\nu \in \mathcal{P}_2(\mathbb{R}^n)$ and $\gamma \in \Gamma(\mu, \nu)$. The following are equivalent:
    \begin{itemize}
        \item[$(1).$] $\nu$ is a probabilistic dual frame of $\mu$ with respect to $\gamma \in \Gamma(\mu, \nu)$.
        \item[$(2).$] For any ${\bf f} \in \mathbb{R}^n$, 
        $${\bf f} = \int_{\mathbb{R}^n \times \mathbb{R}^n} {\bf x} \langle {\bf y}, {\bf f}\rangle d\gamma({\bf x},{\bf y}) \ \text{or} \ {\bf f} = \int_{\mathbb{R}^n \times \mathbb{R}^n} {\bf y} \langle {\bf x}, {\bf f}\rangle d\gamma({\bf x},{\bf y}).$$
        \item[$(3).$] For any ${\bf f}, {\bf g} \in \mathbb{R}^n$, 
        $$\langle {\bf f}, {\bf g} \rangle = \int_{\mathbb{R}^n \times \mathbb{R}^n} \langle {\bf x}, {\bf g} \rangle \langle {\bf y}, {\bf f}\rangle d\gamma({\bf x},{\bf y})  \ \text{or} \  \langle {\bf f}, {\bf g} \rangle= \int_{\mathbb{R}^n \times \mathbb{R}^n} \langle {\bf y}, {\bf g}\rangle \langle {\bf x}, {\bf f} \rangle  d\gamma({\bf x},{\bf y}) .$$ 
    \end{itemize}
\end{lemma}
\begin{proof}
Since $\nu$ is a probabilistic dual frame of $\mu$ with respect to $\gamma \in \Gamma(\mu, \nu)$, then
 \begin{equation*}
       \int_{\mathbb{R}^n \times \mathbb{R}^n} {\bf x} {\bf y}^t   d\gamma({\bf x,y}) =  \int_{\mathbb{R}^n \times \mathbb{R}^n} {\bf y}{\bf x}^t d\gamma({\bf x,y}) = {\bf Id},
    \end{equation*}
which shows that $(1)$ and $(2)$ are equivalent. Clearly $(2)$ implies $(3)$.  And $(3)$ shows that for any ${\bf g} \in \mathbb{R}^n$, 
\begin{equation*}
   \big \langle {\bf f} - \int_{\mathbb{R}^n \times \mathbb{R}^n} {\bf y} \langle {\bf x}, {\bf f}\rangle d\gamma({\bf x},{\bf y}), {\bf g} \big \rangle =   \langle {\bf f}, {\bf g} \rangle - \int_{\mathbb{R}^n \times \mathbb{R}^n} \langle {\bf y}, {\bf g}\rangle  \langle {\bf x}, {\bf f} \rangle d\gamma({\bf x},{\bf y}) = 0.
\end{equation*}
Therefore, $(2)$ follows. 
\end{proof}

From the above lemma, we know that if $\nu$ is a probabilistic dual frame of $\mu$ with respect to some $\gamma \in \Gamma(\mu, \nu)$, then for any ${\bf f} \in \mathbb{R}^n$, 
    \begin{equation*}
         \Vert {\bf f} \Vert^2 = \int_{\mathbb{R}^n \times \mathbb{R}^n} \langle {\bf f}, {\bf x} \rangle \langle {\bf y}, {\bf f}\rangle d\gamma({\bf x, y}).
    \end{equation*}
Conversely, we show that $\nu$ is a probabilistic dual frame of $\mu$ if the above equality holds in a dense subset of $\mathbb{R}^n$, and the proof is placed in Section \ref{ProofOfLemmas}. 
\begin{proposition}\label{denseProperty}
     Let $\mu$ be a probabilistic frame and $D \subset \mathbb{R}^n$ be dense. If there exists $\nu \in \mathcal{P}_2(\mathbb{R}^n)$ and $ \gamma \in \Gamma(\mu, \nu)$ such that $\int_{\mathbb{R}^n \times \mathbb{R}^n} {\bf x}{\bf y}^t d\gamma({\bf x, y})$ is symmetric and
     \begin{equation*}\label{denseDualProperty}
         \Vert {\bf f} \Vert^2 = \int_{\mathbb{R}^n \times \mathbb{R}^n} \langle {\bf f}, {\bf x} \rangle \langle {\bf y}, {\bf f}\rangle d\gamma({\bf x, y}), \  \text{for any ${\bf f} \in D$},
     \end{equation*}
     then $\nu$ is a probabilistic dual frame of $\mu$ with respect to $\gamma \in \Gamma(\mu, \nu)$.
\end{proposition}

Recall that given a measure space $(\Omega, \mu)$, the \textit{essential supremum} of a measurable function $f: \Omega \rightarrow \mathbb{R}$ with respect to $\mu$, denoted by $\mu$-$\text{esssup} \ f$, is the smallest number $\alpha$ such that the set $\{x \in \Omega: f(x)>\alpha\}$ 
has measure zero, and if no such number exists, the essential supremum is taken to be $+\infty$.
By the linearity of matrix trace, we have the following equality and essential supremum inequality that play a role in obtaining the lower bound for the probabilistic dual frame potential.
\begin{lemma}\label{tracelemma}
      Let $\mu$ be a probabilistic frame  and   $\nu$ a probabilistic dual frame to $\mu$ with respect to $\gamma \in \Gamma(\mu, \nu)$. Then 
      \begin{equation*}
         \gamma {\text -} \mathrm{esssup} \ \{ \langle {\bf x, y} \rangle: {\bf x}, {\bf y} \in \mathbb{R}^{n} \} \geq \int_{\mathbb{R}^n \times \mathbb{R}^n} \langle {\bf x, y} \rangle d\gamma({\bf x, y}) = n,
      \end{equation*}
and the equality holds if and only if for $\gamma$ almost all $({\bf x}, {\bf y}) \in \mathbb{R}^n \times \mathbb{R}^n$, $\langle {\bf x, y} \rangle  = n$.
Furthermore, if $p \geq 1$, then  
      \begin{equation*}
         \gamma {\text -} \mathrm{esssup} \ \{ |\langle {\bf x, y} \rangle |^p: {\bf x}, {\bf y} \in \mathbb{R}^{n} \} \geq \int_{\mathbb{R}^n \times \mathbb{R}^n} |\langle {\bf x, y} \rangle | ^p d\gamma({\bf x, y}) \geq n^p,
      \end{equation*}
and when $p>1$, the equality $\gamma {\text -} \mathrm{esssup} \ \{ |\langle {\bf x, y} \rangle |^p: {\bf x}, {\bf y} \in \mathbb{R}^{n} \} = n^p $, or the last equality holds, if and only if  for $\gamma$ almost all $({\bf x}, {\bf y}) \in \mathbb{R}^n \times \mathbb{R}^n$, $\langle {\bf x, y} \rangle  = n$. 
\end{lemma}
\begin{proof}
  Since $\nu$ is a probabilistic dual frame to $\mu$ with respect to $\gamma \in \Gamma(\mu, \nu)$, then
  \begin{equation*}
       \int_{\mathbb{R}^n \times \mathbb{R}^n} {\bf x} {\bf y}^t   d\gamma({\bf x,y}) = {\bf Id}.
    \end{equation*}
    Therefore, by taking trace on both sides, we have
    \begin{equation*}
       n= \trace({\bf Id}) =  \int_{\mathbb{R}^n \times \mathbb{R}^n} \trace({\bf x} {\bf y}^t ) d\gamma({\bf x,y}) = \int_{\mathbb{R}^n \times \mathbb{R}^n} \langle {\bf x}, {\bf y} \rangle d\gamma({\bf x,y}).
    \end{equation*}
    Clearly, we have 
        \begin{equation*}
          \gamma {\text -} \text{esssup} \ \{ \langle {\bf x, y} \rangle: {\bf x}, {\bf y} \in \mathbb{R}^{n} \} \geq \int_{\mathbb{R}^n \times \mathbb{R}^n} \langle {\bf x, y} \rangle d\gamma({\bf x, y}) = n,
      \end{equation*}
      and the equality holds if and only if for $\gamma$ almost all $({\bf x}, {\bf y}) \in \mathbb{R}^n \times \mathbb{R}^n$, $\langle {\bf x, y} \rangle$ is a constant, which is true if and only if the constant is $n$.
    Since $\vert \cdot \vert^p: \mathbb{R} \rightarrow \mathbb{R}$ is convex where $p \geq 1$, then by Jensen's inequality, 
     \begin{equation*}
           \int_{\mathbb{R}^n \times \mathbb{R}^n} |\langle {\bf x, y} \rangle | ^p d\gamma({\bf x, y})  \geq \Big \vert \int_{\mathbb{R}^n \times \mathbb{R}^n} \langle {\bf x, y} \rangle d\gamma({\bf x, y}) \Big \vert ^p = n^p,
      \end{equation*}
and when $p >1$ ($\vert \cdot \vert^p: \mathbb{R} \rightarrow \mathbb{R}$ is strictly convex), the equality holds if and only if for $\gamma$ almost all $({\bf x}, {\bf y}) \in \mathbb{R}^n \times \mathbb{R}^n$, $\langle {\bf x, y} \rangle$ is a constant (equal to $n$). For $p \geq 1$, we further have 
\begin{equation*}
         \gamma {\text -} \text{esssup} \ \{ |\langle {\bf x, y} \rangle |^p: {\bf x}, {\bf y} \in \mathbb{R}^{n} \} \geq \int_{\mathbb{R}^n \times \mathbb{R}^n} |\langle {\bf x, y} \rangle | ^p d\gamma({\bf x, y}) \geq n^p, 
\end{equation*}
and the first equality holds if and only if for $\gamma$ almost all $({\bf x}, {\bf y}) \in \mathbb{R}^n \times \mathbb{R}^n$, $|\langle {\bf x, y} \rangle|$ is constant. Then combining with the last equality condition, we get the desired condition for $\gamma {\text -} \text{esssup} \ \{ |\langle {\bf x, y} \rangle |^p: {\bf x}, {\bf y} \in \mathbb{R}^{n} \} = n^p $ where $p>1$.
\end{proof}

Lemma \ref{tracelemma} shows an analogous result to Proposition 24 in \cite{aceska2022cross}, claiming that if $\{{\bf f}_i\}_{i=1}^N$ is a frame for $\mathbb{R}^n$ and $\{{\bf g}_i\}_{i=1}^N$ is a dual frame of $\{{\bf f}_i\}_{i=1}^N$, then 
\begin{equation*}
    \sum_{i=1}^N |\langle {\bf f}_i, {\bf g}_i \rangle | ^2 \geq \frac{n^2}{N},
\end{equation*}
and the equality holds if and only if $\langle {\bf f}_i, {\bf g}_i \rangle = \frac{n}{N}$, for each $i$. Indeed, if $\mu = \frac{1}{N}\sum_{i=1}^N \delta_{{\bf f}_i}$ is a probabilistic frame and $\nu = \frac{1}{N}\sum_{i=1}^N \delta_{{\bf h}_i}$ a dual frame of $\mu$ with respect to $({\bf Id}, T)_\# \mu \in \Gamma(\mu, \nu)$ where $T:\mathbb{R}^n \rightarrow \mathbb{R}^n$ is such that  $T({\bf f}_i)= {\bf h}_i$ for each $i$, that is to say, $T_\#\mu = \nu$.  Then by Lemma \ref{tracelemma}, 
\begin{equation*}
\begin{split}
       \int_{\mathbb{R}^n \times \mathbb{R}^n} |\langle {\bf x, y} \rangle | ^2 d\gamma({\bf x, y}) = \frac{1}{N} \sum_{i=1}^N |\langle {\bf f}_i, {\bf h}_i \rangle | ^2 
       &\geq \Big \vert \int_{\mathbb{R}^n \times \mathbb{R}^n} \langle {\bf x, y} \rangle d\gamma({\bf x, y}) \Big \vert ^2 =n^2, 
\end{split}  
\end{equation*}
and the equality holds if and only if $\langle {\bf f}_i, {\bf h}_i \rangle = n$, for each $i$. 
For the remaining part of this section, we will focus on a particular type of probabilistic dual frame, namely those that arise as pushforwards of measurable maps.  
\begin{definition}
    Let $\mu$ be a probabilistic frame and let $T: \mathbb{R}^n \rightarrow \mathbb{R}^n$ be measurable such that $T_{\#}\mu \in \mathcal{P}_2(\mathbb{R}^n)$. Then $T_{\#}\mu$ is called a probabilistic dual frame of pushforward type for $\mu$ if 
    \begin{equation*}
        \int_{\mathbb{R}^n} {\bf x} T({\bf x})^t d\mu({\bf x}) = {\bf Id}.
    \end{equation*}
\end{definition}

Indeed, $T_{\#}\mu$ is a dual frame of $\mu$ under $\gamma:=({\bf Id}, T)_{\#}\mu \in \Gamma(\mu, T_{\#}\mu)$, since
    \begin{equation*}
        \int_{\mathbb{R}^n \times \mathbb{R}^n} {\bf x} {\bf y}^t d\gamma({\bf x},{\bf y}) = \int_{\mathbb{R}^n} {\bf x} T({\bf x})^t d\mu({\bf x}) = {\bf Id}.
    \end{equation*}
Clearly, the canonical probabilistic dual ${{\bf S}_\mu^{-1}}_\#\mu$ is of pushforward type. In \cite{wickman2017duality}, the authors showed that given a probabilistic frame $\mu$, ${\psi_h}_\#\mu$ is a probabilistic dual frame of $\mu$, where $h:\mathbb{R}^n \rightarrow \mathbb{R}^n$ is such that $h_{\#}\mu \in \mathcal{P}_2(\mathbb{R}^n)$ and $\psi_h:\mathbb{R}^n \rightarrow \mathbb{R}^n$ is
$$\psi_h({\bf x}) = {{\bf S}_\mu^{-1}}{\bf x}+ h({\bf x}) - \int_{\mathbb{R}^n} \langle {\bf S}_\mu^{-1}{\bf x}, {\bf y} \rangle h({\bf y}) d\mu({\bf y}), \ {\bf x}  \in \mathbb{R}^n.$$
Furthermore, the authors in \cite{chen2025probabilistic} showed that all probabilistic dual frames of pushforward type can be characterized by this construction. However, not every probabilistic dual frame is of pushforward type. For example, $\nu = \frac{1}{2}\delta_{\frac{1}{2}} + \frac{1}{2}\delta_{\frac{3}{2}}$ is a probabilistic dual frame of $\delta_1$ with respect to the product measure $\delta_1 \otimes \nu \in \Gamma(\delta_1, \nu)$, but there does not exist a map $M: \mathbb{R} \rightarrow \mathbb{R}$ such that $\nu = M_\# \delta_1$.

\begin{corollary}[Proposition 5.8 in \cite{chen2025probabilistic}]\label{pushforwardDuals}
      Let $\mu$ be a probabilistic frame and the map $T:\mathbb{R}^n \rightarrow \mathbb{R}^n$ measurable. If $T_\# \mu$ is a probabilistic dual frame of $\mu$, then $T:\mathbb{R}^n \rightarrow \mathbb{R}^n$ precisely satisfies that for any ${\bf x} \in \mathbb{R}^n$, 
      \begin{equation*}
          T({\bf x}) = {{\bf S}_\mu^{-1}}{\bf x} + h({\bf x}) - \int_{\mathbb{R}^n} \langle {\bf S}_\mu^{-1}{\bf x}, {\bf y} \rangle h({\bf y}) d\mu({\bf y}),
      \end{equation*}
where $h:\mathbb{R}^n \rightarrow \mathbb{R}^n$ is such that $h_\# \mu \in \mathcal{P}_2(\mathbb{R}^n)$. 
\end{corollary}

In what follows, we prove that any Bessel probability measure can be transformed into a tight probabilistic frame. Then we show that probabilistic dual frames of pushforward type can be used to characterize tight probabilistic frames. 

\begin{lemma} \label{Bessel_To_Tight}
    Let $\mu$ be a Bessel probability measure with bound $B>0$. Then for any real number $k \geq \frac{1}{2}$, there exists $\nu_k \in \mathcal{P}_2(\mathbb{R}^n)$ such that $\frac{1}{2}\mu+\frac{1}{2}\nu_k$ is a tight probabilistic frame with bound $kB>0$. In particular, there exists $\nu \in \mathcal{P}_2(\mathbb{R}^n)$ such that $\frac{1}{2}\mu+\frac{1}{2}\nu$ is a tight probabilistic frame with bound $B>0$.
\end{lemma}
\begin{proof}
    Let ${\bf S}_\mu$ be the frame operator of $\mu$. Since $\mu$ is Bessel, then ${\bf S}_\mu \leq B {\bf Id}$ and thus $B {\bf Id} - \frac{1}{2k}{\bf S}_\mu$ is positive semi-definite where $k \geq \frac{1}{2}$. By the spectral theorem, $B {\bf Id} - \frac{1}{2k}{\bf S}_\mu$ has a square root $(B {\bf Id} - \frac{1}{2k}{\bf S}_\mu)^{1/2}$. Then, for any ${\bf x} \in \mathbb{R}^n$, 
    $$B{\bf x} =  \frac{1}{2k} {\bf S}_\mu {\bf x} + (B {\bf Id} - \frac{1}{2k}{\bf S}_\mu)^{1/2} (B {\bf Id} - \frac{1}{2k}{\bf S}_\mu)^{1/2}{\bf x}. $$
    Now let $\eta_k \in \mathcal{P}_2(\mathbb{R}^n)$ be any known tight frame with bound $2k$, then 
    $$(B {\bf Id} - \frac{1}{2k}{\bf S}_\mu)^{1/2}{\bf x} =\frac{1}{2k} \int_{\mathbb{R}^n} {\bf y} \big \langle (B {\bf Id} - \frac{1}{2k}{\bf S}_\mu)^{1/2}{\bf x}, {\bf y} \big \rangle d\eta_k({\bf y}).$$
    Therefore, 
     $$ B{\bf x} = \frac{1}{2k} {\bf S}_\mu {\bf x} +  \frac{1}{2k} (B {\bf Id} - \frac{1}{2k}{\bf S}_\mu)^{1/2} \int_{\mathbb{R}^n} {\bf y} \big \langle (B {\bf Id} - \frac{1}{2k}{\bf S}_\mu)^{1/2}{\bf x}, {\bf y} \big \rangle d\eta_k({\bf y}). $$
     Taking inner product with ${\bf x}$ leads to 
     $$ B\|{\bf x}\|^2 =  \frac{1}{2k} \int_{\mathbb{R}^n} |\langle {\bf x}, {\bf y} \rangle|^2 d\mu({\bf y}) +  \frac{1}{2k} \int_{\mathbb{R}^n} \big | \langle {\bf x}, (B {\bf Id} - \frac{1}{2k}{\bf S}_\mu)^{1/2} {\bf y}  \rangle \big|^2 d\eta_k({\bf y}). $$
     Now let $\nu_k := {(B {\bf Id} - \frac{1}{2k}{\bf S}_\mu)^{1/2}}_\#\eta_k \in \mathcal{P}_2(\mathbb{R}^n)$, then for any ${\bf x} \in \mathbb{R}^n$, 
     $$ kB\|{\bf x}\|^2 =  \int_{\mathbb{R}^n} |\langle {\bf x}, {\bf y} \rangle|^2 d \frac{\mu + \nu_k}{2}({\bf y}), $$
     which implies that $\frac{1}{2}\mu+\frac{1}{2}\nu_k$ is a tight probabilistic frame with bound $kB$. We get the last result by letting $k=1$. 
\end{proof}

Note that the restriction $k \ge \tfrac12$ in Lemma \ref{Bessel_To_Tight} is necessary.  For the construction of $\nu_k$ to work, the matrix $B{\bf Id} - \frac{1}{2k}{\bf S}_\mu$ must be positive semi-definite so that one can take its square root to define $\nu_k := {(B {\bf Id} - \frac{1}{2k}{\bf S}_\mu)^{1/2}}_\#\eta_k$. 
Since $\mu$ is Bessel with bound $B$, then ${\bf S}_\mu \le B{\bf Id}$. 
To make $B{\bf Id} - \tfrac{1}{2k}{\bf S}_\mu$ positive semi-definite, we should have $k \ge \frac12$. If $k < \tfrac12$, $B{\bf Id} - \frac{1}{2k}{\bf S}_\mu$ may have negative eigenvalues and its square root may \emph{not} exist,  
which makes the construction of $\nu_k$ impossible. We then have the following characterization of tight probabilistic frames. 

\begin{lemma}
    Let $\mu$ be a probabilistic frame and $k>0$. The following are equivalent:
    \begin{itemize}
        \item[$(1).$] $\mu$ is a tight probabilistic frame with frame bound $\frac{1}{k}$.
        \item[$(2).$] $(k\,\mathrm{Id})_{\#}\mu$ is a probabilistic dual frame of pushforward type for $\mu$.
    \end{itemize}
\end{lemma}
\begin{proof}
    $(1) \implies (2)$ follows by letting the probabilistic dual frame  be the canonical dual of $\mu$. Conversely, if $(2)$ holds, then 
\begin{equation*}
     k\int_{\mathbb{R}^n} {\bf x}{\bf x}^t d\mu({\bf x}) = {\bf Id}.
\end{equation*}
Therefore, ${\bf S}_\mu = \frac{1}{k}{\bf Id}$ and thus $\mu$ is a tight probabilistic frame with bound $\frac{1}{k}$.  
\end{proof}

We have the following trace equality for probabilistic dual frames of pushforward type by Lemma \ref{tracelemma}.  
\begin{lemma}\label{traceEquality}
      Let $\mu$ be a probabilistic frame and $T_{\#}\mu$ a probabilistic dual frame for $\mu$ where $T:\mathbb{R}^n \rightarrow \mathbb{R}^n$ is measurable. Then
       \begin{equation*}
          \int_{\mathbb{R}^n} \langle {\bf x}, T({\bf x}) \rangle d\mu({\bf x})  =  \int_{\mathbb{R}^n} \Vert {\bf S}_\mu^{-1/2} {\bf x} \Vert^2 d\mu({\bf x}) = n.
      \end{equation*}
\end{lemma}
\begin{proof}
   Clearly, the first equality follows from Lemma \ref{tracelemma} by letting $\nu = T_{\#}\mu$ and $\gamma = ({\bf Id}, T)_{\#}\mu$. 
    Since ${{\bf S}_\mu^{-1}}_\# \mu$ is also a probabilistic dual frame, then 
     \begin{equation*}
      n = \int_{\mathbb{R}^n} \langle {\bf x}, {{\bf S}_\mu^{-1}} {\bf x}\rangle d\mu({\bf x}) = \int_{\mathbb{R}^n} \langle {{\bf S}_\mu^{-1/2}}{\bf x}, {{\bf S}_\mu^{-1/2}} {\bf x} \rangle d\mu({\bf x}) = \int_{\mathbb{R}^n} \Vert {\bf S}_\mu^{-1/2}{\bf x} \Vert^2 d\mu({\bf x}).
    \end{equation*}
\end{proof}

It has been shown that if a probability measure is close to a given probabilistic frame in an appropriate sense, then this probability measure is a probabilistic frame, and especially, a sufficient perturbation condition using probabilistic dual frames was given in Theorem 3.6 of \cite{chen2023paley}. Let $\mu$ be a probabilistic frame and $\nu$ a probabilistic dual frame of $\mu$ with respect to $\gamma_{12} \in \Gamma(\mu, \nu)$. Suppose $\eta \in \mathcal{P}_2(\mathbb{R}^n)$ and $\gamma_{23} \in \Gamma(\nu, \eta)$, then by Gluing Lemma {\cite[pp.~59]{figalli2021invitation}}, there exists $ \Tilde{\pi} \in \mathcal{P}(\mathbb{R}^n \times \mathbb{R}^n \times \mathbb{R}^n)$ with marginals $\gamma_{12}$ and $\gamma_{23}$. If
 \begin{equation*}
     \sigma:= \int_{\mathbb{R}^n \times \mathbb{R}^n \times \mathbb{R}^n}  \Vert {\bf x}-{\bf z} \Vert \Vert {\bf y} \Vert d\Tilde{\pi}({\bf x},{\bf y},{\bf z}) <1,
 \end{equation*}
then $\eta$ is a probabilistic frame with bounds $ \frac{(1- \sigma)^2}{M_2(\nu)} \ \text{and} \  M_2(\eta)$. If the dual frame $\nu$ is of pushforward type, i.e., $\nu = T_\#\mu$ for a map $T: \mathbb{R}^n \rightarrow \mathbb{R}^n$, then $\gamma_{12} \in \Gamma(\mu, T_\#\mu)$ and Gluing Lemma implies that $\Tilde{\pi}$ is the product measure $ \gamma_{12} \otimes \eta$. Therefore, the above condition becomes
 \begin{equation*}
    \int_{\mathbb{R}^n \times \mathbb{R}^n}  \Vert {\bf x}-{\bf z} \Vert \Vert T({\bf x}) \Vert d\mu({\bf x}) d\eta({\bf z}) <1.
 \end{equation*}
Indeed, as stated below, the conclusion that $\eta$ is a probabilistic frame holds for any coupling $\gamma \in \Gamma(\mu, \eta)$, besides the product measure $\mu \otimes \eta$,  and its proof is placed in Section \ref{ProofOfLemmas}. The following also generalizes Proposition 3.9 in \cite{chen2023paley} where $T_\#\mu = {{\bf S}_\mu^{-1}}_\# \mu$.

\begin{proposition}\label{dualperturbation}
    Let $\mu$ be a probabilistic frame and let $T:\mathbb{R}^n \rightarrow \mathbb{R}^n$ be such that $ T_\# \mu$ is a probabilistic dual frame to $\mu$. Given $\eta \in \mathcal{P}_2(\mathbb{R}^n)$ and $\gamma \in \Gamma(\mu, \eta)$, if
 \begin{equation*}
     \kappa:= \int_{\mathbb{R}^n \times \mathbb{R}^n}  \Vert {\bf x}-{\bf z} \Vert \Vert T({\bf x}) \Vert d\gamma({\bf x}, {\bf z}) <1,
 \end{equation*}
then $\eta$ is a probabilistic frame with bounds $ \frac{(1- \kappa)^2}{M_2(T_\#\mu)} $ and $ M_2(\eta)$. In particular, if $B>0$ is the upper bound of $T_\#\mu$, then the bounds for $\eta$ are $ \frac{(1- \kappa)^2}{B} $ and $ M_2(\eta)$. 
\end{proposition}

\section{Probabilistic Dual Frame Potentials}\label{dualpotential}

In this section, we study probabilistic dual frame potentials, which are motivated by frame potentials and dual frame potentials of finite frames considered by \cite{benedetto2003finite} and \cite{christensen2020equiangular}, respectively. 
One of the motivations for frame potentials was to measure the total orthogonality of a finite unit-norm frame, and under this interpretation, minimizing the frame potential is equivalent to finding a frame whose elements are as orthogonal to each other as possible \cite{benedetto2003finite}.  Similarly, for a given finite frame, the associated dual frame potential is designed to measure the total cross-orthogonality (or cross-correlation) between the frame and its dual frame. The minimizer is precisely the dual frame whose elements are as orthogonal as possible to the original frame vectors. In particular, the authors in \cite{christensen2020equiangular} showed that the dual frame potential is minimized by the canonical dual frame, meaning that the canonical dual has the least cross-correlation and the highest cross-orthogonality to the given frame. We sketch their result and proof below to make the paper more self-contained.

\begin{theorem}[Theorem 2.2 in \cite{christensen2020equiangular}]\label{FiniteDualPotential}
    Let $\{{\bf f}_i\}_{i=1}^N$ be a frame for $\mathbb{C}^n$ with frame operator ${\bf S}$, and $\{{\bf g}_j\}_{j=1}^N$ a dual frame of $\{{\bf f}_i\}_{i=1}^N$ where $N \geq n$. Then the dual frame potential between $\{{\bf f}_i\}_{i=1}^N$ and $\{{\bf g}_j\}_{j=1}^N$ satisfies
    \begin{equation*}
       \mathrm{DFP}(\{{\bf f}_i\}, \{{\bf g}_j\}): = \sum_{i=1}^N  \sum_{j=1}^N |\langle {\bf f}_i, {\bf g}_j \rangle|^2 \geq n,
    \end{equation*}
    and the equality holds if and only if $\{{\bf g}_j\}_{j=1}^N $ is the canonical dual frame $ \{{\bf S}^{-1}{\bf f}_j\}_{j=1}^N$.
\end{theorem}
\begin{proof}
    It is well-known (see \cite{christensen2016introduction}, Lemma 5.4.2) that for a fixed ${\bf f} \in \mathbb{C}^{n}$, the frame coefficient sequence $\{\langle {\bf f}, {\bf S}^{-1}{\bf f}_j\rangle\}_{i=1}^N$ has the least $\ell^2$-norm energy among any other sequences representing  ${\bf f}$. That is, if ${\bf f} = \sum\limits_{i=1}^N c_i {\bf f}_i$ for some coefficients $\{c_i\}_{i=1}^N$, then 
    \begin{equation}\label{leastDualEnergy}
        \sum_{i=1}^N |c_i|^2 = \sum_{i=1}^N |\langle {\bf f}, {\bf S}^{-1}{\bf f}_i\rangle|^2 + \sum_{i=1}^N |c_i - \langle {\bf f}, {\bf S}^{-1}{\bf f}_i\rangle|^2.
    \end{equation}
    Therefore, 
       $$ \sum_{i=1}^N |c_i|^2 \geq \sum_{i=1}^N |\langle {\bf f}, {\bf S}^{-1}{\bf f}_i\rangle|^2, $$
    and the equality holds if and only if $c_i = \langle {\bf f}, {\bf S}^{-1}{\bf f}_i \rangle$, $1 \leq i \leq N$.  Since $\{{\bf g}_j\}_{j=1}^N$ is a dual frame of $\{{\bf f}_i\}_{i=1}^N$, then 
    ${\bf f} = \sum\limits_{j=1}^N \langle {\bf f}, {\bf g}_j \rangle {\bf f}_j $. Therefore, 
     \begin{equation*}
        \sum_{j=1}^N |\langle {\bf f}, {\bf g}_j \rangle|^2 \geq \sum_{j=1}^N |\langle {\bf f}, {\bf S}^{-1}{\bf f}_j\rangle|^2 ,
    \end{equation*}
    and the equality holds if and only if ${\bf g}_j = {\bf S}^{-1} {\bf f}_j$, $1 \leq j \leq N$. Now let ${\bf f} = {\bf f}_i$ and sum over $i$, we have 
     \begin{equation*}
        \sum_{i=1}^N  \sum_{j=1}^N | \langle {\bf f}_i, {\bf g}_j \rangle|^2 \geq \sum_{i=1}^N  \sum_{j=1}^N |\langle {\bf f}_i, {\bf S}^{-1}{\bf f}_j\rangle|^2 =n ,
    \end{equation*}
    and the equality in the first inequality holds if and only if ${\bf g}_j = {\bf S}^{-1} {\bf f}_j$, $1 \leq j \leq N$. The last identity can be verified by Proposition 17 in \cite{aceska2022cross}. 
\end{proof}

Analogously, given a probabilistic frame $\mu$, one can define the probabilistic dual frame potential (PDFP) to measure the expected cross-orthogonality between the frame and its dual frame. Consequently, minimizing this potential corresponds to selecting a probabilistic dual frame that is, on average, as orthogonal as possible to the original frame.

\begin{definition}
    Let $\mu$ be a probabilistic frame and $\nu \in \mathcal{P}_2(\mathbb{R}^n)$ a probabilistic dual frame to $\mu$. The probabilistic dual frame potential between $\mu$ and $\nu$ is 
          \begin{equation*}
              \mathrm{PDFP}_\mu(\nu) := \int_{\mathbb{R}^n}  \int_{\mathbb{R}^n} |\langle {\bf x}, {\bf y} \rangle| ^2 d\mu({\bf x}) d\nu({\bf y}) .
          \end{equation*}
\end{definition}

Note that if $T_{\#}\mu$ is a probabilistic dual frame for $\mu$ where $T:\mathbb{R}^n \rightarrow \mathbb{R}^n$ is measurable, then the probabilistic dual frame potential between $\mu$ and $T_{\#}\mu$ becomes
     \begin{equation*}
              \mathrm{PDFP}_\mu(T_{\#}\mu) := \int_{\mathbb{R}^n}  \int_{\mathbb{R}^n} \vert \langle {\bf x}, T({\bf y}) \rangle \vert ^2 d\mu({\bf x}) d\mu({\bf y}) .
    \end{equation*}
In addition, we show that the probabilistic dual frame potential for a probabilistic dual pair is invariant under the unitary transform.

\begin{lemma}
    Let $\mu$ be a probabilistic frame and ${\bf U}$ a unitary $n \times n$ matrix.  Then  ${\bf U}_{\#}\mu$ is a probabilistic frame, and if $\nu \in \mathcal{P}_2(\mathbb{R}^n)$ is a probabilistic dual frame to $\mu$, then ${\bf U}_{\#}\nu$ is also a dual frame to ${\bf U}_{\#}\mu$. Furthermore, the probabilistic dual frame potential is unitarily invariant. That is to say,  $\mathrm{PDFP}_\mu(\nu) = \mathrm{PDFP}_{{\bf U}_{\#}\mu}({\bf U}_{\#}\nu)$.

\end{lemma}
\begin{proof}
Since ${\bf U}$ is unitary, then the frame operator of ${\bf U}_{\#}\mu$ is positive definite, i.e.,
     \begin{equation*}
         {\bf S}_{{\bf U}_{\#}\mu} = \int_{\mathbb{R}^n} {\bf y} {\bf y}^{t} d {\bf U}_{\#}\mu({\bf y})  = \int_{\mathbb{R}^n} {\bf U} {\bf x} {{\bf x}}^{t} {\bf U}^t d\mu({\bf x}) = {\bf U} {\bf S}_{\mu} {\bf U}^t >0.
     \end{equation*}
     Therefore, ${\bf U}_{\#}\mu$ is a probabilistic frame. If $\nu$ is a probabilistic dual frame to $\mu$, then by Lemma \ref{DualFrameEquivalence}, there exists $\gamma \in \Gamma(\mu, \nu)$ such that for any ${\bf f}, {\bf g} \in \mathbb{R}^n$, 
        $$\langle {\bf f}, {\bf g} \rangle = \int_{\mathbb{R}^n \times \mathbb{R}^n} \langle {\bf x}, {\bf f} \rangle \langle {\bf y}, {\bf g}\rangle d\gamma({\bf x},{\bf y}).$$ 
Now define $\gamma' := ({\bf U}, {\bf U})_\# \gamma \in \Gamma({\bf U}_{\#}\mu, {\bf U}_{\#}\nu)$. Then for any ${\bf f}, {\bf g} \in \mathbb{R}^n$, 
        \begin{equation*}
        \begin{split}
             \int_{\mathbb{R}^n \times \mathbb{R}^n} \langle {\bf x}, {\bf f} \rangle \langle {\bf y}, {\bf g}\rangle d\gamma'({\bf x},{\bf y}) 
             &= \int_{\mathbb{R}^n \times \mathbb{R}^n} \langle {\bf U} {\bf x}, {\bf f} \rangle \langle {\bf U} {\bf y}, {\bf g}\rangle d\gamma({\bf x},{\bf y}) \\
             &= \int_{\mathbb{R}^n \times \mathbb{R}^n} \langle  {\bf x}, {\bf U}^t {\bf f} \rangle \langle  {\bf y}, {\bf U}^t {\bf g}\rangle d\gamma({\bf x},{\bf y}) \\
             & = \langle {\bf U}^t  {\bf f}, {\bf U}^t  {\bf g} \rangle = \langle {\bf f}, {\bf g} \rangle, 
        \end{split}
        \end{equation*} 
where the last equality comes from ${\bf U}^t$ being unitary.
Therefore, by Lemma \ref{DualFrameEquivalence}  again, ${\bf U}_{\#}\nu$ is  a probabilistic dual frame to ${\bf U}_{\#}\mu$ with respect to $\gamma'$. And the associated probabilistic dual frame potential is unitarily invariant, since
     \begin{equation*}
     \begin{split}
          \mathrm{PDFP}_\mu(\nu) &= \int_{\mathbb{R}^n} \int_{\mathbb{R}^n} |\langle {\bf x}, {\bf y} \rangle|^2 d\mu({\bf x}) d\nu({\bf y}) = \int_{\mathbb{R}^n} \int_{\mathbb{R}^n} |\langle {\bf U} {\bf x}, {\bf U}{\bf y} \rangle|^2 d\mu({\bf x}) d\nu({\bf y}) \\
          &=\int_{\mathbb{R}^n} \int_{\mathbb{R}^n} |\langle {\bf x}, {\bf y} \rangle|^2 d{\bf U}_{\#}\mu({\bf x}) d{\bf U}_{\#}\nu({\bf y}) = \mathrm{PDFP}_{{\bf U}_{\#}\mu}({\bf U}_{\#}\nu).
     \end{split}
     \end{equation*}
\end{proof}

As Equation \ref{leastDualEnergy} suggests, for a fixed ${\bf f} \in \mathbb{R}^{n}$ and a frame $\{{\bf f}_i\}_{i=1}^N$, the frame coefficient sequence $\{\langle {\bf f}, {\bf S}^{-1}{\bf f}_i\rangle\}_{i=1}^N$ has the least $\ell^2$-norm energy among any other sequences representing  ${\bf f}$ (see \cite{christensen2016introduction}, Lemma 5.4.2). We obtain an analogous result for probabilistic frames, which shows that among all reconstructions, the canonical dual frame representation has the least $L^2(\mu, \mathbb{R}^n)$ energy, where $L^2(\mu, \mathbb{R}^n)$ 
denotes the space of square-integrable functions on $\mathbb{R}^n$ with respect to $\mu$. 

\begin{proposition}\label{DualEnergy}
    Let $\mu$ be a probabilistic frame with frame operator ${\bf S}_\mu$. For any fixed ${\bf f} \in \mathbb{R}^n$, if ${\bf f} = \int_{\mathbb{R}^n} {\bf x} \omega({\bf x}) d\mu({\bf x})$ for some $\omega \in L^2(\mu, \mathbb{R}^n)$, then 
   \begin{equation*}
       \int_{\mathbb{R}^n} \omega^2({\bf x}) d\mu({\bf x}) =  \int_{\mathbb{R}^n} |\langle {\bf S}_\mu^{-1}{\bf f}, {\bf x} \rangle|^2  d\mu({\bf x}) + \int_{\mathbb{R}^n} \vert \omega({\bf x})-\langle {\bf S}_\mu^{-1}{\bf f}, {\bf x} \rangle \vert^2 d\mu({\bf x}).
    \end{equation*}
   Furthermore, 
    \begin{equation*}
       \int_{\mathbb{R}^n} \omega^2({\bf x}) d\mu({\bf x}) \geq  \int_{\mathbb{R}^n} |\langle {\bf S}_\mu^{-1}{\bf f}, {\bf x} \rangle|^2  d\mu({\bf x}),
    \end{equation*}
   and the equality holds if and only if $\omega ({\bf x}) = \langle {\bf S}_\mu^{-1}{\bf f}, {\bf x} \rangle$ for $\mu$ almost all ${\bf x} \in \mathbb{R}^n$. 
\end{proposition}

\begin{proof}
For the probabilistic frame $\mu$, the synthesis operator $V:L^2(\mu, \mathbb{R}^n) \rightarrow \mathbb{R}^n$ and its adjoint analysis operator $V^*: \mathbb{R}^n \rightarrow L^2(\mu, \mathbb{R}^n)$ are defined as 
\begin{equation*}
    V(\psi) = \int_{\mathbb{R}^n} {\bf x} \ \psi({\bf x}) d\mu({\bf x}) \in \mathbb{R}^n, \  (V^*{\bf x})(\cdot) = \langle {\bf x}, \cdot \rangle \in L^2(\mu, \mathbb{R}^n).
\end{equation*}
It can be shown that $V$ and $V^*$ are bounded linear and thus
$\Ker V = (\Ran{V^*})^\perp$, where $\Ran V^*$ is the range of the analysis operator $V^*$ and  $\Ker V $  is the kernel of the synthesis operator $V$ given by
\begin{equation*}
    \Ker V = \{\psi \in L^2(\mu, \mathbb{R}^n): V(\psi) = \int_{\mathbb{R}^n} {\bf x} \ \psi({\bf x}) d\mu({\bf x}) = {\bf 0}\}.
\end{equation*}
For the given $\omega \in L^2(\mu, \mathbb{R}^n)$, 
\begin{equation*}
        \omega(\cdot) = \omega(\cdot) - \langle {\bf S}_\mu^{-1}{\bf f}, \cdot \rangle + \langle {\bf S}_\mu^{-1}{\bf f}, \cdot \rangle. 
    \end{equation*}
Since 
\begin{equation*}
    V(\omega - \langle {\bf S}_\mu^{-1}{\bf f}, \cdot \rangle) = \int_{\mathbb{R}^n} {\bf x} \omega({\bf x})  d\mu({\bf x})  - \int_{\mathbb{R}^n} {\bf x} \langle {\bf S}_\mu^{-1}{\bf f},{\bf x} \rangle  d\mu({\bf x}) = {\bf f}-{\bf f}={\bf 0},
\end{equation*}
then $\omega - \langle {\bf S}_\mu^{-1}{\bf f}, \cdot \rangle \in \Ker V = (\Ran{V^*})^\perp$. Since $\langle {\bf S}_\mu^{-1}{\bf f}, \cdot \rangle \in \Ran{V^*}$, then $\omega - \langle {\bf S}_\mu^{-1}{\bf f}, \cdot \rangle $ is orthogonal to $\langle {\bf S}_\mu^{-1}{\bf f}, \cdot \rangle $ in $L^2(\mu, \mathbb{R}^n)$. Therefore, by Pythagorean identity,
\begin{equation*}
        \Vert \omega \Vert_{L^2(\mu, \mathbb{R}^n)}^2 = \Vert \omega - \langle {\bf S}_\mu^{-1}{\bf f}, \cdot \rangle \Vert_{L^2(\mu, \mathbb{R}^n)}^2 + \Vert \langle {\bf S}_\mu^{-1}{\bf f}, \cdot \rangle \Vert_{L^2(\mu, \mathbb{R}^n)}^2.
    \end{equation*}
That is to say, 
\begin{equation*}
       \int_{\mathbb{R}^n} \omega^2({\bf x}) d\mu({\bf x}) =  \int_{\mathbb{R}^n} \vert \omega({\bf x})-\langle {\bf S}_\mu^{-1}{\bf f}, {\bf x} \rangle \vert^2 d\mu({\bf x}) + \int_{\mathbb{R}^n} |\langle {\bf S}_\mu^{-1}{\bf f}, {\bf x} \rangle|^2  d\mu({\bf x}).
    \end{equation*}
 Therefore, 
    \begin{equation*}
       \int_{\mathbb{R}^n} \omega^2({\bf x}) d\mu({\bf x}) \geq  \int_{\mathbb{R}^n} |\langle {\bf S}_\mu^{-1}{\bf f}, {\bf x} \rangle|^2  d\mu({\bf x}),
    \end{equation*}
and the equality holds if and only if
    \begin{equation*}
        \int_{\mathbb{R}^n} \vert \omega({\bf x})-\langle {\bf S}_\mu^{-1}{\bf f}, {\bf x} \rangle \vert^2 d\mu({\bf x}) = 0,
    \end{equation*}
  which is true if and only if $\omega ({\bf x}) = \langle {\bf S}_\mu^{-1}{\bf f}, {\bf x} \rangle$ for $\mu$ almost all ${\bf x} \in \mathbb{R}^n$.
\end{proof}

We then show that for a given probabilistic frame $\mu$ and its dual frame $T_{\#}\mu$ where $T: \mathbb{R}^n \rightarrow \mathbb{R}^n$ is measurable, their dual frame potential is minimized if and only if $T({\bf y}) =  {\bf S}_\mu^{-1}{\bf y}$ for $\mu$ almost all ${\bf y} \in \mathbb{R}^n$, which means that the probabilistic dual frame $T_{\#}\mu$ is indeed the canonical dual ${{\bf S}_\mu^{-1}}_\#\mu$ up to a $\mu$-null set. Therefore, the minimizer of the probabilistic dual frame potential among probabilistic dual frames of pushforward type is the canonical dual.

\begin{theorem}\label{DualFramePotential}
    Suppose $\mu$ is a probabilistic frame with frame operator ${\bf S}_\mu$ and   $T_{\#}\mu$ a probabilistic dual frame to $\mu$ where $T: \mathbb{R}^n \rightarrow \mathbb{R}^n$ is a measurable map. Then 
    \begin{equation*}
              \mathrm{PDFP}_\mu(T_{\#}\mu):= \int_{\mathbb{R}^n}  \int_{\mathbb{R}^n}  |\langle {\bf x}, T({\bf y}) \rangle|^{2} d\mu({\bf x}) d\mu({\bf y}) \geq n,
    \end{equation*}
     and the equality holds if and only if $T({\bf y}) =  {\bf S}_\mu^{-1}{\bf y}$ for $\mu$ almost all ${\bf y} \in \mathbb{R}^n$.
\end{theorem}
\begin{proof}
Since $T_{\#}\mu$ is a probabilistic dual frame to $\mu$, then for any ${\bf x} \in \mathbb{R}^n$, 
    \begin{equation*}
        {\bf x} = \int_{\mathbb{R}^n} {\bf y} \langle {\bf x}, T({\bf y}) \rangle  d\mu({\bf y}).
    \end{equation*}
    Note that $\langle {\bf x}, T(\cdot) \rangle: \mathbb{R}^n \rightarrow  \mathbb{R} \in L^2(\mu, \mathbb{R}^n)$, since
    \begin{equation*}
        \int_{\mathbb{R}^n} \vert \langle {\bf x}, T({\bf y}) \rangle \vert^2  d\mu({\bf y}) \leq \Vert {\bf x} \Vert^2    \int_{\mathbb{R}^n} \Vert T({\bf y}) \Vert^2  d\mu({\bf y}) < +\infty. 
    \end{equation*}
    Then, by Proposition \ref{DualEnergy}, we know that for any fixed ${\bf x} \in \mathbb{R}^n$,
    \begin{equation}\label{DualEnergy2}
         \int_{\mathbb{R}^n}  |\langle {\bf x}, T({\bf y}) \rangle|^2  d\mu({\bf y}) 
          \geq \int_{\mathbb{R}^n}  |\langle {\bf S}_\mu^{-1}{\bf x}, {\bf y} \rangle|^2  d\mu({\bf y}) = \int_{\mathbb{R}^n}  |\langle {\bf x}, {\bf S}_\mu^{-1}{\bf y} \rangle|^2  d\mu({\bf y}),
    \end{equation}
and the equality holds if and only if $T({\bf y}) =  {\bf S}_\mu^{-1}{\bf y}$ for $\mu$ almost all ${\bf y} \in \mathbb{R}^n$. Thus 
 \begin{equation*}
     \begin{split}
         \int_{\mathbb{R}^n} \int_{\mathbb{R}^n}  |\langle {\bf x}, T({\bf y}) \rangle|^2   d\mu({\bf y}) d\mu({\bf x}) 
       &\geq \int_{\mathbb{R}^n} \int_{\mathbb{R}^n}  |\langle {\bf x}, {\bf S}_\mu^{-1}{\bf y} \rangle|^2  d\mu({\bf y}) d\mu({\bf x}) \\
       &=  \int_{\mathbb{R}^n} \int_{\mathbb{R}^n}  |\langle {\bf S}_\mu^{-1/2}{\bf x}, {\bf S}_\mu^{-1/2}{\bf y} \rangle|^2 d\mu({\bf x}) d\mu({\bf y}) \\
       &= \int_{\mathbb{R}^n} \Vert {\bf S}_\mu^{-1/2}{\bf y} \Vert^2 d\mu({\bf y}) = n,
     \end{split}
\end{equation*}
where the second step is due to the symmetry of ${\bf S}_\mu^{-1/2}$, the third step follows from that ${{\bf S}_\mu^{-1/2}}_\#\mu$ is a Parseval frame, and the last identity comes from Lemma \ref{traceEquality}. The equality clearly holds if $T({\bf y}) =  {\bf S}_\mu^{-1}{\bf y}$ for $\mu$ almost all ${\bf y} \in \mathbb{R}^n$. Conversely, if the equality holds, then the equality in the first step must hold. Thus by Equation \ref{DualEnergy2}, we know that for $\mu$ almost all ${\bf x} \in \mathbb{R}^n$, 
 \begin{equation*}
       \int_{\mathbb{R}^n}  |\langle {\bf x}, T({\bf y}) \rangle|^2  d\mu({\bf y}) = \int_{\mathbb{R}^n}  |\langle {\bf x}, {\bf S}_\mu^{-1}{\bf y} \rangle|^2  d\mu({\bf y}).
\end{equation*}
In particular, at least for some ${\bf x}_0 \in \supp(\mu)$, we have
 \begin{equation*}
       \int_{\mathbb{R}^n}  |\langle {\bf x}_0, T({\bf y}) \rangle| ^2  d\mu({\bf y}) = \int_{\mathbb{R}^n} |\langle {\bf x}_0, {\bf S}_\mu^{-1}{\bf y} \rangle| ^2  d\mu({\bf y}),
\end{equation*}
which implies that $T({\bf y}) =  {\bf S}_\mu^{-1}{\bf y}$ for $\mu$ almost all ${\bf y} \in \mathbb{R}^n$ by the equality condition in Equation \ref{DualEnergy2}.
\end{proof}

We have the following corollary about the probabilistic $2p$-dual frame potential of pushforward type, where $p > 1$ and $|\supp(\mu)|$ is the cardinality of $\supp(\mu)$.

\begin{corollary}\label{2pDualPotential}
    Let $p > 1$. Suppose $\mu$ is a probabilistic frame and $T_{\#}\mu \in \mathcal{P}_2(\mathbb{R}^n)$ is a probabilistic dual frame of pushforward type for $\mu$ where $T: \mathbb{R}^n \rightarrow \mathbb{R}^n$. Then 
    \begin{equation*}
              \int_{\mathbb{R}^n}  \int_{\mathbb{R}^n} \vert \langle {\bf x}, T({\bf y}) \rangle \vert^{2p} d\mu({\bf x}) d\mu({\bf y}) \geq n^p,
    \end{equation*}
     where the equality does not hold when $n \geq 2$ or $|\supp(\mu) |\geq 3$. Furthermore, if $n \geq 2$ (or $|\supp(\mu) |\geq 3$) and $p \geq 1$, 
          \begin{equation*}
              \mu \otimes \mu {\text -} \mathrm{esssup} \ \{ |\langle {\bf x}, T({\bf y}) \rangle |^{2p}: {\bf x}, {\bf y} \in \mathbb{R}^{n}\} > n^p.
          \end{equation*}
\end{corollary}
    \begin{proof} 
Since $\vert \cdot \vert^p: \mathbb{R} \rightarrow \mathbb{R}$ is strictly convex where $p > 1$, then by Jensen's inequality and Theorem \ref{DualFramePotential}, we have
     \begin{equation*}
          \int_{\mathbb{R}^n}  \int_{\mathbb{R}^n} \vert \langle {\bf x}, T({\bf y}) \rangle \vert^{2p} d\mu({\bf x}) d\mu({\bf y}) \geq \Big \vert  \int_{\mathbb{R}^n}  \int_{\mathbb{R}^n} \vert \langle {\bf x}, T({\bf y}) \rangle \vert ^2 d\mu({\bf x}) d\mu({\bf y})  \Big \vert^p \geq n^p.
      \end{equation*}
Furthermore, the equality in the first inequality holds if and only if for $\mu \otimes \mu$ almost all $({\bf x,y}) \in \mathbb{R}^n \times \mathbb{R}^n$,  $\vert \langle {\bf x}, T({\bf y}) \rangle  \vert$ is constant, 
and the equality in the second inequality holds if and only if for $\mu$ almost all ${\bf y} \in \mathbb{R}^n$, $T({\bf y}) =  {\bf S}_\mu^{-1}{\bf y}$.  
Therefore, the equality holds if and only if  for $\mu \otimes \mu$ almost all $({\bf x, y}) \in \mathbb{R}^n \times \mathbb{R}^n$, $\vert \langle {\bf x}, {\bf S}_\mu^{-1}{\bf y} \rangle  \vert = \sqrt{n}$. Since ${\bf S}_\mu^{-1}$ is positive definite, then one can define an inner product $\langle \cdot, \cdot \rangle_\mu: \mathbb{R}^n \times \mathbb{R}^n \rightarrow \mathbb{R}$ given by $\langle {\bf x}, {\bf y} \rangle_\mu = {\bf x}^t{\bf S}_\mu^{-1} {\bf y}$, and the reduced norm is given by 
$$\|{\bf x}\|_\mu = \sqrt{\langle {\bf x}, {\bf x} \rangle_\mu}.$$
Therefore, if $n \geq 2$ or $|\supp(\mu) |\geq 3$ and the equality holds,  then for $\mu \otimes \mu$ almost all $({\bf x, y}) \in \mathbb{R}^n \times \mathbb{R}^n$, $|\langle {\bf x}, {\bf y} \rangle_\mu|=\sqrt{n}$. Since the continuous map $({\bf x,y}) \mapsto |\langle {\bf x},{\bf S}_\mu^{-1}{\bf y}\rangle|$ is constant $\mu\otimes\mu$-almost everywhere, it must be constant on the support $\supp(\mu \otimes \mu) = \supp(\mu)\times\supp(\mu)$. Hence, for any ${\bf x, y} \in \supp(\mu)$, $\vert \langle {\bf x}, {\bf y} \rangle_\mu  \vert = \sqrt{n}$, $\vert \langle {\bf y}, {\bf y} \rangle_\mu  \vert = \sqrt{n}$ and $\vert \langle {\bf x}, {\bf x} \rangle_\mu  \vert = \sqrt{n}$.  
Then, $\|{\bf x}\|_\mu \|{\bf y}\|_\mu =\sqrt{n}= \vert \langle {\bf x}, {\bf y} \rangle_\mu  \vert $ and thus by Cauchy-Schwarz  inequality, there exists $c \in \mathbb{R}$ such that ${\bf y} = c{\bf x}$, which implies $|c| \|{\bf x}\|_\mu^2 = \sqrt{n}$ implies that $c=1$ or $c=-1$. Hence, $\supp(\mu)$ has at most two points, either $\supp(\mu) = \{{\bf x}\}$ or $\supp(\mu) = \{{\bf x},-{\bf x}\}$ for some ${\bf x} \in \mathbb{R}^n$, and thus the linear span of $\supp(\mu)$ is a one-dimensional subspace, which contradicts the assumption that $|\supp(\mu) |\geq 3$ or $\mu$ is a probabilistic frame when $n \geq 2$. Therefore, the equality does not hold when $n \geq 2$ or $|\supp(\mu) |\geq 3$. Furthermore, when $n \geq 2$ and $p>1$ or $|\supp(\mu) |\geq 3$ and $p>1$ , we have
     \begin{equation*}
              \mu \otimes \mu {\text -} \text{esssup} \ \{ |\langle {\bf x}, T({\bf y}) \rangle |^{2p}: {\bf x}, {\bf y} \in \mathbb{R}^{n}\}  \geq  \int_{\mathbb{R}^n}  \int_{\mathbb{R}^n} \vert \langle {\bf x}, T({\bf y}) \rangle \vert ^{2p} d\mu({\bf x}) d\mu({\bf y}) > n^p.
    \end{equation*}
For the $p=1$ case, i.e., $n \geq 2$ and $p=1$ or $|\supp(\mu) |\geq 3$ and $p=1$,  we have
    \begin{equation*}
              \mu \otimes \mu {\text -} \text{esssup} \ \{ |\langle {\bf x}, T({\bf y}) \rangle |^{2}: {\bf x}, {\bf y} \in \mathbb{R}^{n}\}  \geq  \int_{\mathbb{R}^n}  \int_{\mathbb{R}^n} \vert \langle {\bf x}, T({\bf y}) \rangle \vert ^{2} d\mu({\bf x}) d\mu({\bf y}) \geq n.
    \end{equation*}
Similarly, both equalities hold if and only if  for $\mu \otimes \mu$ almost all $({\bf x, y}) \in \mathbb{R}^n \times \mathbb{R}^n$, $\vert \langle {\bf x}, {\bf S}_\mu^{-1}{\bf y} \rangle  \vert$ is the constant $ \sqrt{n}$. Using a similar argument as above, we claim that if the equality holds, $\supp(\mu)$ must have at most two points and the linear span of $\supp(\mu)$ is one-dimensional, which contradicts the assumption that $|\supp(\mu) |\geq 3$ or $\mu$ is a probabilistic frame when $n \geq 2$. This completes the proof. 
\end{proof}

Finally, we arrive at the last main result of this paper. For a given probabilistic frame with bounds $A$ and $B$, we claim that the probabilistic dual frame potential is bounded below by $n \frac{A}{B}$. Since $A \leq B$, then this lower bound satisfies
$$
n \frac{A}{B} \leq n,
$$
where $n$ is precisely the minimal value of the probabilistic dual frame potential attained among all probabilistic dual frames of pushforward type in Theorem \ref{DualFramePotential}.  This is because there exist probabilistic dual frames of non-pushforward type, and the minimization of probabilistic dual frame potential in Theorem \ref{dualframepotential2} is among a larger set compared to the case in Theorem \ref{DualFramePotential}. However,  the effect of probabilistic dual frames of non-pushforward type becomes negligible when the given probabilistic frame is tight ($A=B$), since these two lower bounds are both $n$.

\begin{theorem}\label{dualframepotential2}
    Let $\mu$ be a probabilistic frame with frame bounds $0<A \leq B$ and let $\nu \in \mathcal{P}_2(\mathbb{R}^n)$ be a probabilistic dual frame to $\mu$. Then
          \begin{equation*}
              \mathrm{PDFP}_\mu(\nu) := \int_{\mathbb{R}^n}  \int_{\mathbb{R}^n} |\langle {\bf x}, {\bf y} \rangle |^2 d\mu({\bf x}) d\nu({\bf y}) \geq  \frac{A}{B} n,
          \end{equation*}
    and the equality holds if and only if $\mu$ is a tight probabilistic frame and $\nu = {{\bf S}^{-1}_\mu}_\# \mu$. 
\end{theorem}
\begin{proof}
Since $\mu$ is a probabilistic frame with bounds $A$ and $B$, then
    \begin{equation*}
              \mathrm{PDFP}_\mu(\nu) := \int_{\mathbb{R}^n}  \int_{\mathbb{R}^n} |\langle {\bf x}, {\bf y} \rangle| ^2 d\mu({\bf x}) d\nu({\bf y}) \geq A \int_{\mathbb{R}^n}  \Vert {\bf y}  \Vert ^2  d\nu({\bf y}). 
          \end{equation*}
Since $\nu$ is a probabilistic dual frame to $\mu$, then by Lemma \ref{tracelemma}, there exists $\gamma \in \Gamma(\mu, \nu)$ such that
      \begin{equation*}
          \int_{\mathbb{R}^n \times \mathbb{R}^n} \langle {\bf x, y} \rangle d\gamma({\bf x, y}) = n.
      \end{equation*}
Then, using Cauchy-Schwarz inequality twice, we have
 \begin{equation}\label{CauchySchwarz }
 \begin{split}
      n = \int_{\mathbb{R}^n \times \mathbb{R}^n} \langle {\bf x, y} \rangle d\gamma({\bf x, y}) 
      & \leq \int_{\mathbb{R}^n \times \mathbb{R}^n} \| {\bf x} \| \| {\bf  y} \| d\gamma({\bf x, y}) \\
      & \leq \sqrt{\int_{\mathbb{R}^n}  \Vert {\bf x}  \Vert ^2  d\mu({\bf x}) \int_{\mathbb{R}^n}  \Vert {\bf y}  \Vert ^2  d\nu({\bf y})}.
 \end{split}
 \end{equation}
Therefore, we have
$$\int_{\mathbb{R}^n}  \Vert {\bf y}  \Vert ^2  d\nu({\bf y}) \geq \frac{n^2}{\int_{\mathbb{R}^n}  \Vert {\bf x}  \Vert ^2  d\mu({\bf x})} \geq \frac{n^2}{nB} = \frac{n}{B}, $$
and the last inequality above follows from the following fact:
\begin{equation} \label{eqn:2nd_moment}
    0<\int_{\mathbb{R}^n}  \Vert {\bf x}  \Vert ^2  d\mu({\bf x}) = \sum_{i=1}^n \int_{\mathbb{R}^n}   |\langle {\bf e}_i, {\bf x} \rangle|^2  d\mu({\bf x}) \leq \sum_{i=1}^n B \Vert{\bf e}_i \Vert^2 =  nB,
\end{equation} 
where $\{{\bf e}_i\}_{i=1}^n$ is the standard orthonormal basis in $\mathbb{R}^n$. 
Therefore, 
 \begin{equation*}
              \mathrm{PDFP}_\mu(\nu) = \int_{\mathbb{R}^n}  \int_{\mathbb{R}^n}  |\langle {\bf x}, {\bf y} \rangle|^2 d\mu({\bf x}) d\nu({\bf y})  \geq A \int_{\mathbb{R}^n}  \Vert {\bf y}  \Vert ^2  d\nu({\bf y}) \geq n \frac{A}{B}.
\end{equation*}

For the equality, if $\mu$ is a tight frame with bound $A>0$ and $\nu = {{\bf S}^{-1}_\mu}_\# \mu$, then ${\bf S}^{-1}_\mu = \frac{1}{A} {\bf Id}$ and thus
 \begin{equation*}
              \mathrm{PDFP}_\mu(\nu) = \frac{1}{A^2}\int_{\mathbb{R}^n}  \int_{\mathbb{R}^n}  |\langle {\bf x}, {\bf y} \rangle|^2 d\mu({\bf x}) d\mu({\bf y})  = \frac{1}{A} \int_{\mathbb{R}^n}  \Vert {\bf y}  \Vert ^2  d\mu({\bf y}) = n,
\end{equation*}
where the last two equalities follow from the tightness of $\mu$ and Equation \ref{eqn:2nd_moment}.

Conversely, if the equality holds, then for $\nu$ almost all ${\bf y}$, 
 \begin{equation*}
      \int_{\mathbb{R}^n} |\langle {\bf x}, {\bf y} \rangle| ^2 d\mu({\bf x})  = A  \Vert {\bf y}  \Vert ^2,
\end{equation*}
and 
$$\int_{\mathbb{R}^n}  \Vert {\bf y}  \Vert ^2  d\nu({\bf y}) = \frac{n^2}{\int_{\mathbb{R}^n}  \Vert {\bf x}  \Vert ^2  d\mu({\bf x})}  = \frac{n}{B},$$
which implies that the equalities in the above two Cauchy-Schwarz inequalities in Equation \ref{CauchySchwarz } must hold. Thus, for the first equality, we can claim that for $\gamma$ almost all $({\bf x, y}) \in \mathbb{R}^n \times \mathbb{R}^n$, there exists a constant $c_{\bf x} \geq 0$ that may be related to ${\bf x}$ such that ${\bf y} = c_{\bf x} {\bf x}$.
Then, $\gamma$ is supported on the graph of the map $T: \mathbb{R}^n \rightarrow \mathbb{R}^n$ where $T ({\bf{x}}) = c_{\bf{x}}{\bf{x}}$, that is to say, $\gamma = ({\bf{Id}}, T)_\#\mu$. For the second equality in Equation \ref{CauchySchwarz }, we claim that there exists a constant $c \geq 0$ such that $\Vert \cdot \Vert_{{\bf y}}= c \Vert \cdot  \Vert_{\bf x}$ in the sense that $\Vert \cdot \Vert_{{\bf y}}$ and $\Vert \cdot \Vert_{{\bf x}}$ are linearly dependent vectors in $L^2(\gamma, \mathbb{R}^n \times \mathbb{R}^n)$ where $\Vert ({\bf x, y}) \Vert_{{\bf y}} = \Vert {\bf y} \Vert$ and $\Vert ({\bf x, y}) \Vert_{{\bf x}} = \Vert {\bf x} \Vert$. 
Therefore, for $\gamma$ almost all $ ({\bf x}, {\bf y}) \in \mathbb{R}^n \times \mathbb{R}^n$, 
$${\bf y} = c_{\bf x} {\bf x} \ \text{and} \ \Vert {\bf y} \Vert= c \Vert {\bf x}  \Vert. $$
If $\bf x$ is nonzero, then $c_{\bf x}=c$ and if $\bf x =0$, redefining $c_{\bf x} =c$ doesn't affect the validity of the above two equations. Thus, for $\gamma$ almost all $ ({\bf x}, {\bf y}) \in \mathbb{R}^n \times \mathbb{R}^n$, $c_{\bf x}=c$, independent of $\bf x$, and ${\bf y} = c {\bf x}$. Then $\gamma = ({\bf Id}, c{\bf Id})_\#\mu \in \Gamma(\mu, (c{\bf Id})_\#\mu)$. Since $\gamma \in \Gamma(\mu, \nu)$, then $\nu = (c{\bf Id})_\#\mu$, which is a probabilistic dual frame of $\mu$ with respect to $\gamma = ({\bf Id}, c{\bf Id})_\#\mu$. Then
$${\bf Id} = \int_{\mathbb{R}^n \times \mathbb{R}^n} {\bf x} {\bf y}^t d\gamma({\bf x, y}) = c \int_{\mathbb{R}^n} {\bf x} {\bf x}^t d\mu({\bf x}) =c {\bf S}_{\mu},$$
which implies $c>0$ (otherwise $c=0$ implies $ {\bf Id} = {\bf 0}_{n \times n}$) and ${\bf S}_{\mu} = \frac{1}{c}{\bf Id}$. 
Hence, $\mu$ is a tight probabilistic frame with bound $\frac{1}{c}$ and $\nu = (c{\bf Id})_\#\mu= {{\bf S}^{-1}_\mu}_\# \mu$.
\end{proof}

Theorem \ref{dualframepotential2} says that the lower bound of probabilistic dual frame potential is saturated if and only if the given probabilistic frame is tight and the probabilistic dual frame is canonical, which is a ``mixture'' of the equality conditions for the frame potential and dual frame potential in Section \ref{introduction}. However, when the given probabilistic frame is tight, the equality conditions in Theorem \ref{DualFramePotential} and Theorem \ref{dualframepotential2} coincide: the equality holds if and only if the probabilistic dual frame is the canonical dual. We further have the following corollary about the probabilistic $2p$-dual frame potential where $p > 1$.

\begin{corollary}\label{pPotential}
       Let $p > 1$. Suppose $\mu$ is a probabilistic frame with frame bounds $0<A \leq B$  and $\nu$ is a probabilistic dual frame to $\mu$, then
          \begin{equation*}
               \int_{\mathbb{R}^n}  \int_{\mathbb{R}^n} \vert \langle {\bf x}, {\bf y} \rangle \vert ^{2p} d\mu({\bf x}) d\nu({\bf y}) \geq n^p \frac{A^p}{B^p},
          \end{equation*}
where the equality does not hold when $n \geq 2$ or $|\supp(\mu) |\geq 3$. Furthermore, if $n \geq 2$ (or $|\supp(\mu) |\geq 3$) and $p \geq 1$, 
          \begin{equation*}
              \mu \otimes \nu {\text -} \text{esssup} \ \{|\langle {\bf x, y} \rangle |^{2p}: {\bf x, y} \in \mathbb{R}^n\} > n^p \frac{A^p}{B^p}.
          \end{equation*}
\end{corollary}
\begin{proof}
Since $\vert \cdot \vert^p: \mathbb{R} \rightarrow \mathbb{R}$ is strictly convex when $p > 1$, then by Jensen's inequality and Theorem \ref{dualframepotential2}, we have
     \begin{equation*}
          \int_{\mathbb{R}^n}  \int_{\mathbb{R}^n} \vert \langle {\bf x}, {\bf y} \rangle \vert^{2p} d\mu({\bf x}) d\nu({\bf y}) \geq \Big \vert  \int_{\mathbb{R}^n}  \int_{\mathbb{R}^n} \vert \langle {\bf x}, {\bf y} \rangle \vert ^2 d\mu({\bf x}) d\nu({\bf y})  \Big \vert^p \geq n^p \frac{A^p}{B^p}.
      \end{equation*}
Furthermore,  the equality in the first inequality holds if and only if for $\mu \otimes \nu$ almost all $({\bf x, y}) \in \mathbb{R}^n \times \mathbb{R}^n$, $\vert \langle {\bf x}, {\bf y} \rangle  \vert$ is constant, and the equality in the second inequality holds if and only if $\mu$ is a tight probabilistic frame and $\nu = {{\bf S}^{-1}_\mu}_\# \mu$. Then both equalities hold if and only if ${\bf S}_\mu = A {\bf Id}$, $\nu = (\frac{1}{A} {\bf Id})_\# \mu $, and for $\mu \otimes \nu$ almost all $({\bf x, y}) \in \mathbb{R}^n \times \mathbb{R}^n$, $\vert \langle {\bf x}, {\bf y} \rangle  \vert = \sqrt{n}$.
Since the continuous map $({\bf x,y}) \mapsto |\langle {\bf x}, {\bf y}\rangle|$ is constant $\mu\otimes\nu$-almost everywhere, it must be constant on its support $\supp(\mu \otimes \nu) = \supp(\mu)\times\supp(\nu)$.
Note that if ${\bf x} \in \supp(\mu)$, then $\frac{1}{A}{\bf x} \in \supp(\nu)$. Therefore, if $|\supp(\mu) |\geq 3$ (or $n \geq 2$) and the equality holds, then for any ${\bf x, y} \in \supp(\mu)$, $\vert \langle {\bf x}, \frac{1}{A}{\bf y} \rangle  \vert = \sqrt{n}$, $\vert \langle {\bf x}, \frac{1}{A}{\bf x} \rangle  \vert = \sqrt{n}$ and $\vert \langle {\bf y}, \frac{1}{A}{\bf y} \rangle  \vert = \sqrt{n}$, implying  $\| {\bf x} \| = \| {\bf y} \|  = \sqrt{A} n^{1/4}$ and thus $\|{\bf x}\| \|{\bf y}\| =A \sqrt{n}= \vert \langle {\bf x}, {\bf y} \rangle  \vert $. By Cauchy-Schwarz  inequality, we must have ${\bf y} =  {\bf x}$ or ${\bf y} = - {\bf x}$. Therefore, $\supp(\mu)$ has at most two points, either $\supp(\mu) = \{{\bf x}\}$ or $\supp(\mu) = \{{\bf x},-{\bf x}\}$, and the linear span of $\supp(\mu)$ is one-dimensional, which contradicts the assumption that $|\supp(\mu) |\geq 3$ or $\mu$ is a probabilistic frame when $n \geq 2$. Therefore, the equality does not hold when $n \geq 2$ or $|\supp(\mu) |\geq 3$. Moreover, when $n \geq 2$ (or $|\supp(\mu) |\geq 3$) and $p>1$, we have
     \begin{equation*}
              \mu \otimes \nu {\text -} \text{esssup} \ \{|\langle {\bf x, y} \rangle |^{2p}: {\bf x, y} \in \mathbb{R}^n\} \geq  \int_{\mathbb{R}^n}  \int_{\mathbb{R}^n} \vert \langle {\bf x}, {\bf y} \rangle \vert ^{2p} d\mu({\bf x}) d\nu({\bf y}) > n^p \frac{A^p}{B^p}.
          \end{equation*}
For the $p=1$ case, i.e. $n \geq 2$ (or $|\supp(\mu) |\geq 3$) and $p=1$, we have
 \begin{equation*}
               \mu \otimes \nu {\text -} \text{esssup} \ \{|\langle {\bf x, y} \rangle |^{2}: {\bf x, y} \in \mathbb{R}^n\} \geq  \int_{\mathbb{R}^n}  \int_{\mathbb{R}^n} \vert \langle {\bf x}, {\bf y} \rangle \vert ^{2} d\mu({\bf x}) d\nu({\bf y}) \geq n \frac{A}{B}.
\end{equation*}
Similarly, both equalities hold if and only if for $\mu \otimes \nu$ almost all $({\bf x, y}) \in \mathbb{R}^n \times \mathbb{R}^n$, $\vert \langle {\bf x}, {\bf y} \rangle  \vert=\sqrt{n}$, ${\bf S}_\mu = A {\bf Id}$, and $\nu = (\frac{1}{A} {\bf Id})_\# \mu $. And we get an analogous contradiction by using the same argument above, if the equality holds. 
\end{proof}

We finish this section by considering the case when $n=1$ and $|\supp(\mu)|  \leq 2$. When $n=1$ and $\supp(\mu)=\{z\}$ where $z \neq 0$, we claim that both equalities in the probabilistic $p$-dual frame potential where $p > 1$ and the associated supremum dual frame potential hold if and only if the probabilistic dual frame is the canonical dual. In addition, when $n=1$ and  $\supp(\mu)=\{z, -z\}$, we claim that the canonical dual frame condition for equality saturation holds again in the probabilistic $2p$-dual frame potential and the associated supremum dual frame potential where $p \geq 1$.
 \begin{example}\label{OneSupport}
 \normalfont
     Note that any $\mu = \delta_z$ where $z \neq 0 \in \mathbb{R}$ is a tight probabilistic frame for the real line. And the set of probabilistic dual frames of $\mu = \delta_z$ is the collection of probability measures on $\mathbb{R}$ with mean value $\frac{1}{z}$ and finite second moments. Indeed, let $\nu \in \mathcal{P}_2(\mathbb{R})$ be a probabilistic dual frame to $\mu$ with respect to $\gamma \in \Gamma(\mu, \nu)$. Since $\mu=\delta_z$ is the delta measure at $z$, then $\gamma \in \Gamma(\mu, \nu)$ is unique and is given by the product measure $\gamma = \mu \otimes \nu$. Therefore, 
    \begin{equation*}
        \int_{\mathbb{R}} \int_{\mathbb{R}} xy d\mu(x) d\nu(y) = z \int_{\mathbb{R} } y d\nu(y) =1.
    \end{equation*}
     Hence, the mean of $\nu$ is $\frac{1}{z}$. Note that 
     $|\cdot|^p: \mathbb{R} \rightarrow \mathbb{R}$ is convex where $p \geq 1$, then by Jensen's inequality, we have 
    \begin{equation*}
       \int_{\mathbb{R}}  \int_{\mathbb{R}} |xy|^{p} d\mu({x}) d\nu({y}) \geq  \left|\int_{\mathbb{R}}  \int_{\mathbb{R}} xy d\mu({x}) d\nu({y}) \right|^p = 1,
    \end{equation*}
and when $p > 1$ ($|\cdot|^p: \mathbb{R} \rightarrow \mathbb{R}$ is strictly convex), the equality holds if and only if for $\mu \otimes \nu$ almost all $(x,y) \in \mathbb{R} \times \mathbb{R}$, $xy=1$. Since $\mu = \delta_z$, this means
$y = \frac{1}{z}$ $\nu$-almost everywhere, which is equivalent to
$\nu = {{ S}_\mu^{-1}}_\#{\mu} = \delta_{\frac{1}{z}}$. Thus, when $p > 1$,
    \begin{equation*}
       \int_{\mathbb{R}}  \int_{\mathbb{R}} |xy|^{p} d\mu({x}) d\nu({y}) \geq   1,
    \end{equation*}
and the equality holds if and only if 
$\nu = {{ S}_\mu^{-1}}_\#{\mu} = \delta_{\frac{1}{z}}$.  
Furthermore, when $p \geq 1$, 
\begin{equation*}
     \mu \otimes \nu {\text -} \text{esssup} \ \{|x y|^{p}: x, y \in \mathbb{R}\} \geq  \int_{\mathbb{R}}  \int_{\mathbb{R}} |xy|^{p} d\mu({x}) d\nu({y}) \geq 1,
\end{equation*}
and when $p>1$, both equalities hold if and only if for $\mu \otimes \nu$ almost all $(x, y) \in \mathbb{R} \times \mathbb{R}$, $|xy|=1$, and $\nu = {{ S}_\mu^{-1}}_\#{\mu} = \delta_{\frac{1}{z}}$. Since $\mu = \delta_z$ and $\nu = \delta_{\frac{1}{z}}$, the first condition is unnecessary and thus when $p > 1$, 
\begin{equation*}
      \mu \otimes \nu {\text -} \text{esssup} \ \{|x y|^{p}: x, y \in \mathbb{R}\}  \geq 1,
\end{equation*}
and the equality holds if and only if $\nu = {{S}_\mu^{-1}}_\#{\mu} = \delta_{\frac{1}{z}}$.
 \end{example}

 \begin{example}
 \normalfont
    Let $0<w<1$ and $\mu = w \delta_z +(1-w) \delta_{-z}$ where $z \neq 0 \in \mathbb{R}$. Clearly, $\mu$ is a tight probabilistic frame for the real line.
    Now let $p > 1$ and $\nu$ be a probabilistic dual frame for $\mu$. Since
     $|\cdot|^p: \mathbb{R} \rightarrow \mathbb{R}$ is strictly convex, then by Jensen's inequality and Theorem \ref{dualframepotential2}, 
    \begin{equation*}
       \int_{\mathbb{R}}  \int_{\mathbb{R}} |xy|^{2p} d\mu({x}) d\nu({y}) \geq    \left|\int_{\mathbb{R}}\int_{\mathbb{R}} x^{2}y^{2} d\mu({x}) d\nu({y}) \right|^p  \geq 1,
    \end{equation*}
and both equalities hold if and only if for $\mu \otimes \nu$ almost all $(x, y) \in \mathbb{R} \times \mathbb{R}$, $|xy|=1$, and $\nu = {{S}_\mu^{-1}}_\#{\mu} = w \delta_{\frac{1}{z}} +(1-w) \delta_{-\frac{1}{z}}$, which is equivalent to $\nu =  {{S}_\mu^{-1}}_\#{\mu}$ since $\mu = w \delta_z +(1-w) \delta_{-z}$. By Theorem \ref{dualframepotential2}, we still have the above result when $p=1$. Therefore, when $p \geq 1$, we have
 \begin{equation*}
       \int_{\mathbb{R}}  \int_{\mathbb{R}} |xy|^{2p} d\mu({x}) d\nu({y}) \geq 1,
    \end{equation*}
and the equality holds if and only if  $\nu =  {{S}_\mu^{-1}}_\#{\mu} = w \delta_{\frac{1}{z}} +(1-w) \delta_{-\frac{1}{z}}$.
Similarly, 
\begin{equation*}
     \mu \otimes \nu {\text -} \text{esssup} \ \{|xy|^{2p}: x, y \in \mathbb{R}\} \geq  \int_{\mathbb{R}}  \int_{\mathbb{R}} |xy|^{2p} d\mu({x}) d\nu({y}) \geq 1,
\end{equation*}
and both equalities hold again if and only if for $\mu \otimes \nu$ almost all $(x, y) \in \mathbb{R} \times \mathbb{R}$, $|xy|=1$, and $\nu = {{S}_\mu^{-1}}_\#{\mu} = w \delta_{\frac{1}{z}} +(1-w) \delta_{-\frac{1}{z}}$. Therefore, when $p \geq 1$, after dropping the unnecessary constant condition, we have
\begin{equation*}
     \mu \otimes \nu {\text -} \text{esssup} \ \{|xy|^{2p}: x, y \in \mathbb{R}\} \geq 1,
\end{equation*}
and the equality holds if and only if $\nu = {{S}_\mu^{-1}}_\#{\mu} = w \delta_{\frac{1}{z}} +(1-w) \delta_{-\frac{1}{z}}$.
 \end{example}

\section{Concluding Questions}\label{section4}
In this section, we conclude the paper with the last proposition, which can be used to rewrite the probabilistic (dual) frame potential in terms of the $p$-Wasserstein distance $W_p(\cdot, \cdot)$. This may provide a new clue to the minimization of many kinds of frame potentials from the perspectives of optimal transport and Wasserstein distance.  We use $\pi_{{\bf x}^{\perp}}$ to denote the orthogonal projection onto the plane ${\bf x}^{\perp}$ of vectors perpendicular to ${\bf x} \in \mathbb{R}^n$, and $(\pi_{{\bf x}^{\perp}})_{\#}$ the associated pushforward on measures.
\begin{proposition}[Proposition 1.2 in \cite{chen2025probabilistic}]\label{Was_Hyperplane_Frame}
Suppose $\mu \in \mathcal{P}_p(\mathbb{R}^n)$ with $p \geq 1$. Then for any point ${\bf x}$ on the unit sphere $\mathbb{S}^{n-1}$, we have
$$W_p^p(\mu, (\pi_{{\bf x}^{\perp}})_{\#}\mu)=\int_{\mathbb{R}^n} |\langle{\bf x}, {\bf y} \rangle |^p d\mu({\bf y}).$$
\end{proposition}

If the probabilistic frame is supported on the unit sphere $\mathbb{S}^{n-1}$, the authors in \cite{chen2025probabilistic} rewrote the minimization problem of probabilistic $p$-frame potential as
\begin{equation*}
  \underset{\mu \in \cP(\mathbb{S}^{n-1})}{\inf}  \iint_{(\mathbb{S}^{n-1})^2}  \vert \left\langle {\bf x},{\bf y} \right\rangle \vert^p d\mu({\bf y}) d\mu({\bf x}) = \underset{\mu \in \cP(\mathbb{S}^{n-1})}{\inf}   \int_{\mathbb{S}^{n-1}} W^p_p(\mu, (\pi_{{\bf x}^{\perp}})_{\#}\mu) \ d\mu({\bf x}).
\end{equation*}
Significant progress has been made on this question when $p>0$ is even \cite{ehler2012minimization, wickman2023gradient, bilyk2022optimal}. And when $n \geq 2$ and $p>0$ is not even, the authors in \cite{bilyk2022optimal} conjectured that the optimizer is a finite discrete measure on the unit sphere $\mathbb{S}^{n-1}$.  

This paper gives a lower bound for the probabilistic $p$-dual frame potential in Theorem \ref{dualframepotential2} and Corollary \ref{pPotential} where $p \geq 1$ is even. And when $p \geq 1$ is non-even, the problem is generally open. 
However, if the frame $\mu$ is supported on $\mathbb{S}^{n-1}$,  the probabilistic $p$-dual frame potential where $p \geq 1$ is not even can be written as 
\begin{equation*}
 \int_{\mathbb{S}^{n-1}} \int_{\mathbb{R}^n}  |\left\langle {\bf x},{\bf y} \right\rangle|^p d\nu({\bf y}) d\mu({\bf x}) =   \int_{\mathbb{S}^{n-1}} W^p_p(\nu, (\pi_{{\bf x}^{\perp}})_{\#}\nu) \ d\mu({\bf x}),
\end{equation*}
which may shed new light on the lower bound of this probabilistic $p$-dual frame potential from the perspectives of optimal transport and Wasserstein distance.

\section{Proofs of Proposition \ref{denseProperty} and Proposition \ref{dualperturbation}}\label{ProofOfLemmas}
\begin{proof}[Proof of Proposition \ref{denseProperty}]
Note that for any ${\bf f}, {\bf g} \in D$, 
      \begin{equation*}
      \begin{split}
          \langle {\bf f}, {\bf g} \rangle & = \frac{1}{2} (\Vert {\bf f}+{\bf g} \Vert^2 -\Vert {\bf f}\Vert^2 - \Vert {\bf g} \Vert^2) \\
         &=  \frac{1}{2} \int_{\mathbb{R}^n \times \mathbb{R}^n} \langle {\bf f+g}, {\bf x} \rangle \langle {\bf y}, {\bf f+g}\rangle - \langle {\bf f}, {\bf x} \rangle \langle {\bf y}, {\bf f}\rangle -  \langle {\bf g}, {\bf x} \rangle \langle {\bf y}, {\bf g}\rangle d\gamma({\bf x, y}) \\
         &= \frac{1}{2} \int_{\mathbb{R}^n \times \mathbb{R}^n} \langle {\bf f}, {\bf x} \rangle \langle {\bf y}, {\bf g}\rangle + \langle {\bf g}, {\bf x} \rangle \langle {\bf y}, {\bf f} \rangle d\gamma({\bf x, y}).
      \end{split}
     \end{equation*}
     Since $\int_{\mathbb{R}^n \times \mathbb{R}^n} {\bf x}{\bf y}^t d\gamma({\bf x, y})$ is symmetric,  then for any ${\bf f}, {\bf g} \in D$, 
    \begin{equation*}
         \langle {\bf f}, {\bf g} \rangle = \int_{\mathbb{R}^n \times \mathbb{R}^n} \langle {\bf f}, {\bf x} \rangle \langle {\bf y}, {\bf g}\rangle d\gamma({\bf x},{\bf y}). 
     \end{equation*}
     Since $D$ is dense in $\mathbb{R}^n$, then for any ${\bf f}, {\bf g} \in \mathbb{R}^n$, there exist $\{ {\bf f}_i \}$, $\{{\bf g}_j\} \subset D$ such that ${\bf f}_i \rightarrow {\bf f}$ and ${\bf g}_j \rightarrow {\bf g}$. Then
     \begin{equation*}
     \begin{split}
         \langle {\bf f}, {\bf g} \rangle = \lim_{i \rightarrow \infty} \lim_{j \rightarrow \infty} \langle {\bf f}_i, {\bf g}_j \rangle &= \lim_{i \rightarrow \infty} \lim_{j \rightarrow \infty}
 \int_{\mathbb{R}^n \times \mathbb{R}^n} \langle {\bf f}_i, {\bf x} \rangle \langle {\bf y}, {\bf g}_j\rangle d\gamma({\bf x,y}) \\
    &= 
 \int_{\mathbb{R}^n \times \mathbb{R}^n} \lim_{i \rightarrow \infty} \lim_{j \rightarrow \infty}  \langle {\bf f}_i, {\bf x} \rangle \langle {\bf y}, {\bf g}_j\rangle d\gamma({\bf x,y}) \\
 &= 
 \int_{\mathbb{R}^n \times \mathbb{R}^n} \langle {\bf f}, {\bf x} \rangle \langle {\bf y}, {\bf g}\rangle d\gamma({\bf x,y}). \\
     \end{split}
     \end{equation*}
Thus, by Lemma \ref{DualFrameEquivalence},  $\nu$ is a probabilistic dual frame of $\mu$ with respect to $\gamma \in \Gamma(\mu, \nu)$. Note that we can exchange the limit and integral in the above, since for any ${\bf p} \in \mathbb{R}^n$, $$\int_{\mathbb{R}^n \times \mathbb{R}^n} \langle {\bf p}, {\bf x} \rangle \langle {\bf y}, \cdot\rangle d\gamma({\bf x, y}): \mathbb{R}^n \rightarrow \mathbb{R}$$ 
is a bounded linear operator:  for any ${\bf z} \in \mathbb{R}^n$, 
\begin{equation*}
\begin{split}
    \Big |\int_{\mathbb{R}^n \times \mathbb{R}^n} \langle {\bf p}, {\bf x} \rangle \langle {\bf y}, {\bf z} \rangle d\gamma({\bf x, y}) \Big |^2
   & \leq \int_{\mathbb{R}^n}  |\langle {\bf p}, {\bf x} \rangle|^2 d\mu({\bf x}) \int_{\mathbb{R}^n}  |\langle {\bf z}, {\bf y}\rangle|^2 d\nu({\bf y}) \\
   & \leq \Vert {\bf p} \Vert^2 M_2(\mu) M_2(\nu) \  \Vert {\bf z} \Vert^2.
\end{split}
\end{equation*}
Similarly, for any fixed ${\bf q}$ in $\mathbb{R}^n$, the following linear operator
$$\int_{\mathbb{R}^n \times \mathbb{R}^n} \langle \cdot, {\bf x} \rangle \langle {\bf y, q}\rangle d\gamma({\bf x, y}): \mathbb{R}^n \rightarrow \mathbb{R}$$ 
is also bounded.
\end{proof}

 \begin{proof}[Proof of Proposition \ref{dualperturbation}]
Since $\eta \in \mathcal{P}_2(\mathbb{R}^n)$, then $\eta$ is Bessel with bound  $ M_2(\eta)$. Next, let us show the lower frame bound. Define a linear operator $L: \mathbb{R}^n \rightarrow \mathbb{R}^n$ by 
     \begin{equation*}
         L({\bf f}) = \int_{\mathbb{R}^n \times \mathbb{R}^n} \left\langle {\bf f}, T({\bf x}) \right\rangle {\bf z}  d\gamma({\bf x, z}), \ \text{for any} \ {\bf f} \in \mathbb{R}^n.
     \end{equation*}
Since $ T_\# \mu$ is a probabilistic dual frame to $\mu$ and $\gamma \in \Gamma(\mu, \eta)$, then 
     \begin{equation*}
    {\bf f} = \int_{\mathbb{R}^n} \left\langle {\bf f}, T({\bf x}) \right\rangle {\bf x}  d\mu({\bf x})= \int_{\mathbb{R}^n \times \mathbb{R}^n } \left\langle {\bf f}, T({\bf x}) \right\rangle {\bf x}   d\gamma({\bf x, z}) , \ \text{for any} \ {\bf f} \in \mathbb{R}^n.
\end{equation*}
Therefore, 
\begin{equation*}
\begin{split}
    \Vert {\bf f} -L({\bf f}) \Vert &=  \Big \Vert \int_{\mathbb{R}^n \times \mathbb{R}^n } \left\langle {\bf f}, T({\bf x}) \right\rangle {\bf x}   d\gamma({\bf x, z})- \int_{\mathbb{R}^n \times \mathbb{R}^n} \left\langle {\bf f}, T({\bf x}) \right\rangle {\bf z}  d\gamma({\bf x, z}) \Big \Vert \\
    &= \Big \Vert  \int_{\mathbb{R}^n \times \mathbb{R}^n } \left\langle {\bf f}, T({\bf x}) \right\rangle ({\bf x} -{\bf z})   d\gamma({\bf x, z}) \Big \Vert \leq  \kappa \Vert {\bf f}  \Vert < \Vert {\bf f}  \Vert.
\end{split}
\end{equation*}
Thus, $L: \mathbb{R}^n \rightarrow \mathbb{R}^n$ is invertible and $\Vert L^{-1} \Vert \leq \frac{1}{1-\kappa} $. Then for any ${\bf f} \in \mathbb{R}^n$, 
\begin{equation*}
    {\bf f} = LL^{-1}({\bf f}) = \int_{\mathbb{R}^n \times \mathbb{R}^n} \left\langle L^{-1}{\bf f}, T({\bf x}) \right\rangle {\bf z}  d\gamma({\bf x, z}).
\end{equation*}
Therefore, 
\begin{equation*}
\begin{split}
   \Vert  {\bf f} \Vert^4 &=  |\left\langle {\bf f}, {\bf f} \right\rangle|^2 = \Big \vert \int_{\mathbb{R}^n \times \mathbb{R}^n} \left\langle L^{-1}{\bf f}, T({\bf x}) \right\rangle \left\langle {\bf f}, {\bf z} \right\rangle d\gamma({\bf x, z}) \Big \vert^2 \\
   & \leq \int_{\mathbb{R}^n \times \mathbb{R}^n} |\left\langle L^{-1}{\bf f}, T({\bf x}) \right\rangle|^2 d\gamma({\bf x, z})  \ \int_{\mathbb{R}^n \times \mathbb{R}^n} |\left\langle {\bf f}, {\bf z} \right\rangle|^2 d\gamma({\bf x, z})\\
    & = \int_{\mathbb{R}^n} |\left\langle L^{-1}{\bf f}, {\bf x} \right\rangle|^2 dT_\#\mu({\bf x})  \ \int_{\mathbb{R}^n} |\left\langle {\bf f}, {\bf z} \right\rangle|^2 d\eta({\bf z})\\
  &\leq B \Vert L^{-1} {\bf f}\Vert^2  \ \int_{\mathbb{R}^n} |\left\langle {\bf f}, {\bf z} \right\rangle|^2 d\eta({\bf z}) \leq  \frac{B\Vert {\bf f} \Vert^2}{(1- \kappa)^2}  \ \int_{\mathbb{R}^n} |\left\langle {\bf f}, {\bf z} \right\rangle|^2 d\eta({\bf z}),
\end{split}
\end{equation*}
where the first inequality is due to Cauchy-Schwarz inequality and the second inequality follows from the frame property of $T_\#\mu$. 
Thus, for any ${\bf f} \in \mathbb{R}^n$, 
\begin{equation*}
  \frac{(1- \kappa)^2}{B}    \Vert  {\bf f} \Vert^2  \leq  \int_{\mathbb{R}^n} |\left\langle {\bf f}, {\bf z} \right\rangle|^2 d\eta({\bf z}) \leq  M_2(\eta) \Vert  {\bf f} \Vert^2.
\end{equation*} 
Therefore, $\eta$ is a probabilistic frame with bounds $ \frac{(1- \kappa)^2}{B} \ \text{and} \ M_2(\eta)$. Since $M_2(T_\#\mu)$ is always an upper bound for $T_\#\mu$, then $\eta$ is a probabilistic frame with bounds
\begin{equation*}
    \frac{(1- \kappa)^2}{M_2(T_\#\mu)} \ \text{and} \ M_2(\eta).
\end{equation*}
\end{proof}

{\bf Data Availability}: No data were used in this study.

{\bf Funding}: The author received no funding for this work.

{\bf Acknowledgment}: The author would like to thank the reviewers for their constructive comments.

\section*{Declarations}
{\bf Competing Interests}: The author declares no competing interests.

%%%%%%%%%%%%%%%%%%%%%%%%%%%%%%%%%% References %%%%%%%%%%%%%%%%%%%%%%%%%%%%%%%

\end{document}